\documentclass[leqno]{amsart}
\usepackage{amsfonts,amssymb,amsmath,amsgen,amsthm}
\usepackage{hyperref}

\theoremstyle{plain}
\newtheorem{theorem}{Theorem}[section]
\newtheorem{definition}[theorem]{Definition}

\newtheorem{lemma}[theorem]{Lemma}

\newtheorem{proposition}[theorem]{Proposition}
\newtheorem{hyp}[theorem]{Assumption}
\theoremstyle{remark}
\newtheorem{remark}[theorem]{Remark}
\newtheorem*{notation}{Notation}

\def\R{{\mathbf R}}
\def\N{{\mathbf N}}
\def\Sch{{\mathcal S}}
\def\O{\mathcal O}

\def\({\left(}
\def\){\right)}
\def\<{\left\langle}
\def\>{\right\rangle}
\def\le{\leqslant}
\def\ge{\geqslant}

\def\Eq#1#2{\mathop{\sim}\limits_{#1\rightarrow#2}}
\def\Tend#1#2{\mathop{\longrightarrow}\limits_{#1\rightarrow#2}}

\def\d{{\partial}}
\def\eps{\varepsilon}
\def\l{\lambda}
\def\om{\omega}
\def\si{{\sigma}}

\DeclareMathOperator{\IM}{Im}

\numberwithin{equation}{section}

\begin{document}

\title[Nonlinear coherent states and
  Ehrenfest time]{Nonlinear coherent states and
  Ehrenfest time for  Schr\"odinger equation}
\author[R. Carles]{R\'emi Carles}
\address[R. Carles]{Univ. Montpellier~2\\Math\'ematiques
\\CC~051\\F-34095 Montpellier}
\address{CNRS, UMR 5149\\  F-34095 Montpellier\\ France}
\email{Remi.Carles@math.cnrs.fr}
\author[C. Fermanian]{Clotilde~Fermanian-Kammerer}
\address[C. Fermanian]{LAMA UMR CNRS 8050,
Universit\'e Paris EST\\
61, avenue du G\'en\'eral de Gaulle\\
94010 Cr\'eteil Cedex\\ France}
\email{Clotilde.Fermanian@univ-paris12.fr}
\begin{abstract}
  We consider the propagation of wave packets for
  the nonlinear Schr\"odinger equation, in the semi-classical
  limit. We establish the existence of a critical size for the initial
  data, in terms of the Planck constant: if the initial data are too
  small, the nonlinearity is negligible up to the Ehrenfest time. If
  the initial data have the critical size, then at leading order the
  wave function propagates 
  like a coherent state whose envelope is given by a nonlinear
  equation, up to a time of the same order as the Ehrenfest time. We
  also prove a nonlinear superposition principle for these nonlinear
  wave packets. 
\end{abstract}
\thanks{This work was supported by the French ANR project
  R.A.S. (ANR-08-JCJC-0124-01)}  
\maketitle

\section{Introduction}
\label{sec:intro}
We consider semi-classical limit $\eps\to 0$ for the nonlinear
Schr\"odinger equation 
\begin{equation}\label{eq:NLS0}
 \left\{
\begin{aligned}
     i\eps\d_t \psi^\eps +\frac{\eps^2}{2}\Delta \psi^\eps&=V(x)\psi^\eps
     +\l |\psi^\eps |^{2\si} \psi^\eps, \quad
(t,x)\in
    {\R}_+\times {\R}^d, \\
     \psi^\eps_{\mid t=0} &= \psi^\eps_{0},
  \end{aligned}
\right.
\end{equation}
with  $\l\in\R$, $d\ge 1$. The nonlinearity is  energy
subcritical  ($\sigma< 2/(d-2)$ if $d\ge 3$). This equation arises for
instance as a model for Bose--Einstein Condensation, where, among
other possibilities, $V$  may be
exactly a harmonic potential, or a truncated harmonic potential (hence
not exactly quadratic); see e.g. \cite{DGPS,JP01}.   
\smallbreak

Assuming that $V$ is smooth and subquadratic (this notion is made precise
below, see Assumption~\ref{hyp:V}), we know that for each $\eps>0$,
\eqref{eq:NLS0} has a 
unique global solution in the energy space  
$$\Sigma=\left\{f\in H^1(\R^d),\;\;x\mapsto |x|f(x)\in L^2(\R^d) \right\},$$
provided $\psi_0^\eps\in \Sigma$ and, either $\sigma<2/d$, or
$\left(\sigma\ge 2/d \;\;{\rm and }\;\; \lambda\ge 0\right)$, while if
$\l\ge 2/d$ and $\l<0$, finite time blow-up may occur; see
\cite{Ca-p}. We assume that the initial data $\psi^\eps_0$ is a
localized  wave packet of the form
\begin{equation}\label{data0}
\psi^\eps_0(x)=\eps^\beta\times \eps^{-d/4} a\left(\frac{x-x_0}{\sqrt\eps}\right)
e^{i{(x-x_0)\cdot\xi_0/\eps}},\quad a\in{\mathcal S}({\R}^d),\quad
x_0,\xi_0\in \R^d.
\end{equation}
Such data, which are called \emph{semi-classical wave packets} (or
\emph{coherent states}),
have been extensively 
studied in the linear case (see
e.g. \cite{BR01,CR97,CR06,L86,P97}). In particular, 
Gaussian wave packets  are used in numerical simulation of quantum
chemistry like Initial Value Representations methods. On this subject,
the reader can refer to the recent papers \cite{Rob-p,Rou-p,SR09} 
 where overview and  references on the
topics can be found. These methods rely on the fact that if the data is a 
wave packet, then the solution of the linear equation ($\l=0$)
associated with \eqref{eq:NLS0} still is a wave packet at
leading order up to times of order $C\log 
 \left(\frac{1}{\eps}\right)$: such a large (as $\eps\to 0$) time is
 called \emph{Ehrenfest time}, see e.g. \cite{BGP99,HJ00,HJ01}.  Our
 aim here is to  investigate what remains of these 
facts in the nonlinear case ($\lambda\not=0$). 
\smallbreak

In the present nonlinear setting, a new parameter has to be
considered: the size of the initial data, hence the factor
$\eps^\beta$ in \eqref{data0}. The goal of this paper is to justify a
notion of criticality for $\beta$: for
$\beta>\beta_c:=1/(2\sigma)+d/4$, the initial 
data are too small to ignite the nonlinearity at leading order, and
the leading order behavior of $\psi^\eps$ as $\eps\to 0$ is the same
as in the linear case $\l=0$, up to Ehrenfest time. On the other hand,
if $\beta =\beta_c$, 
the function $\psi^\eps$ is given at leading order by a wave packet
whose envelope satisfies a \emph{nonlinear} equation, up to a
nonlinear analogue of the Ehrenfest time. We show moreover 
a nonlinear superposition principle: when the initial data is the sum
of two wave packets of the form \eqref{data0}, then $\psi^\eps$ is
approximated at leading order by the sum of the approximations
obtained in the case of a single initial coherent state. 
\smallbreak

Up to changing $\psi^\eps$ to $\eps^{-\beta}\psi^\eps$, we may assume
that the initial data are of order $\O(1)$ in $L^2(\R^d)$, and we
consider
\begin{equation}
  \label{eq:NLS}
   \left\{
\begin{aligned}
     i\eps\d_t \psi^\eps +\frac{\eps^2}{2}\Delta \psi^\eps&=V(x)\psi^\eps
     +\l\eps^{\alpha} |\psi^\eps |^{2\si} \psi^\eps, \quad
(t,x)\in
    {\R}_+\times {\R}^d, \\
     \psi^\eps(0,x) &= \eps^{-d/4} a\left(\frac{x-x_0}{\sqrt\eps}\right)
e^{i{(x-x_0)\cdot\xi_0/\eps}},
  \end{aligned}
\right.
\end{equation}
where $\alpha =2\beta\sigma$.

\subsection{The linear case}
\label{sec:lin1}
In this paragraph, we assume $\l=0$:
\begin{equation}
  \label{eq:Schrod-lin}
    i\eps\d_t \psi^\eps +\frac{\eps^2}{2}\Delta \psi^\eps=V(x)\psi^\eps
    \quad;\quad 
     \psi^\eps(0,x) = \eps^{-d/4} a\left(\frac{x-x_0}{\sqrt\eps}\right)
e^{i{(x-x_0)\cdot\xi_0/\eps}}.
\end{equation}
 The assumption we make on the external potential throughout
this paper (even when $\l\not =0$) is the
following:
\begin{hyp}\label{hyp:V}
  The external potential $V$ is smooth, real-valued, and subquadratic:
  \begin{equation*}
 V\in C^\infty(\R^d;\R)\quad \text{and}\quad   \d_x^\gamma V\in
 L^\infty\(\R^d\),\quad \forall |\gamma|\ge 2.  
  \end{equation*}
\end{hyp}
Consider the classical trajectories associated with
the Hamiltonian $\frac{|\xi|^2}{2}+V(x)$:
\begin{equation}\label{eq:traj}
 \dot x(t)=\xi(t),\;\;\dot \xi(t)=-\nabla V(x(t));\quad x(0)=x_0,\;\xi(0)=\xi_0.
 \end{equation}
 These trajectories satisfy
 $$\frac{|\xi(t)|^2}{ 2 }+V(x(t))=\frac{|\xi_0|^2}{2} +V(x_0),\quad
 \forall t\in \R.$$
 The fact that the potential is subquadratic implies that the
 trajectories grow at most exponentially in time.
 
 \begin{notation}
 For two positive numbers $a^\eps$ and $b^\eps$,
the notation $ a^\eps\lesssim b^\eps$
means that there exists $C>0$ \emph{independent of} $\eps$ such that
for all $\eps\in ]0,1]$, $a^\eps\le Cb^\eps$.  
\end{notation}
 
 \begin{lemma}\label{lem:traj}
Let $(x_0,\xi_0)\in \R^d\times \R^d$. Under Assumption~\ref{hyp:V},
\eqref{eq:traj} has a unique global, smooth solution 
$(x,\xi)\in 
C^\infty(\R;\R^d)^2$. It grows at
most exponentially:    
 \begin{equation}\label{growthtraj}
 \exists C_0>0,\quad \left|x(t)\right|+\left|\xi(t)\right|\lesssim
 {\rm e}^{C_0t},\quad \forall t\in \R.
 \end{equation}
 \end{lemma}
 \begin{proof}[Sketch of the proof] We explain the exponential control
   only. We infer from \eqref{eq:traj} that $x$ solves an Hamiltonian
   ordinary differential equation,
   \begin{equation*}
     \ddot x(t)+\nabla V\(x(t)\)=0. 
   \end{equation*}
 Multiply this equation by $\dot x(t)$,
 \begin{equation*}
   \frac{d}{dt}\(\(\dot x\)^2 +V\(x(t)\)\)=0,
 \end{equation*}
 and notice that in view of
 Assumption~\ref{hyp:V}, $V(y)\lesssim \<y\>^2$:
 \begin{equation*}
   \dot x(t) \lesssim \<x(t)\>,
 \end{equation*}
and the estimate follows.
\end{proof}
\begin{remark}
  The case $V(x)=-|x|^2$ shows that  the result of
  Lemma~\ref{lem:traj} is sharp. 
\end{remark}
 We associate with these trajectories  the \emph{classical action}
\begin{equation}\label{eq:action}
S(t)=\int_0^t \left( \frac{1}{2} |\xi(s)|^2-V(x(s))\right)\,ds.
\end{equation}
We observe that if we change the unknown function $\psi^\eps$ to
$u^\eps$ by
\begin{equation}
  \label{eq:chginc}
  \psi^\eps(t,x)=\eps^{-d/4} u^\eps 
\left(t,\frac{x-x(t)}{\sqrt\eps}\right)e^{i\left(S(t)+\xi(t)\cdot
    (x-x(t))\right)/\eps},
\end{equation}
then, in terms of $u^\eps=u^\eps(t,y)$, \eqref{eq:Schrod-lin} is equivalent 
\begin{equation}\label{eq:ueps0}
i\d_t
u^\eps+\frac{1}{2}\Delta u^\eps=V^\eps(t,y)
u^\eps\quad ;\quad u^\eps(0,y) = a(y),
\end{equation}
where the external time-dependent potential $V^\eps$ is given by
\begin{equation}
  \label{eq:Veps}
  V^\eps(t,y)= \frac{1}{\eps}\(V(x(t)+
\sqrt{\eps}y)-V(x(t))-\sqrt{\eps}\<\nabla V(x(t)),y\>\).
\end{equation}
This expression reveals the first terms of the Taylor expansion of $V$
about the point $x(t)$. Passing formally to the limit, $V^\eps$
converges to the Hessian of $V$ at $x(t)$ evaluated at $(y,y)$. One
does not even need to pass to the limit if $V$ is a polynomial of
degree at most two: in that case, we see that the solution $\psi^\eps$
remains exactly a coherent state for all time. 
Let us denote by $Q(t)$ the symmetric matrix
$$Q(t)={\rm Hess}\,V(x(t)).$$
It is well-known that if $v$ solves
\begin{equation}\label{eq:vlinear}
i\d_tv+\frac{1}{2}\Delta v=\frac{1}{2}\< Q(t)y,y\> v\quad
;\quad v(0,y)=a(y),
\end{equation}
then the function
\begin{equation}
  \label{eq:philin}
  \varphi^\eps_{\rm lin}(t,x)=\eps^{-d/4} v
\left(t,\frac{x-x(t)}{\sqrt\eps}\right)e^{i\left(S(t)+\xi(t)\cdot
    (x-x(t))\right)/\eps}
\end{equation}
approximates $\psi^\eps$ for large time in the sense that there exists
$C>0$ independent of $\eps$ such that
\begin{equation*}
  \|\psi^\eps(t,\cdot)-\varphi^\eps_{\rm
  lin}(t,\cdot)\|_{L^2(\R^d)}\le C\sqrt\eps e^{Ct}.
\end{equation*}
See e.g. \cite{BGP99,CR97,CR06,CR07,H80,HJ00,HJ01} and references therein. We
give a short proof of this estimate, which can be considered as the
initial step toward the nonlinear analysis which is presented in the
next paragraph. We first notice that since $V$ is subquadratic, we
have the following pointwise estimate:
\begin{equation}\label{eq:restepot}
 \left\lvert  V^\eps(t,y) -\frac{1}{2}\< Q(t)y,y\>\right\rvert\le \underline
 C 
 \sqrt \eps |y|^3,
\end{equation}
for some constant $\underline C$ independent of $t$. The error
$r^\eps_{\rm lin}=u^\eps-v$ satisfies
\begin{equation*}
  i \d_t r^\eps_{\rm lin} +\frac{1}{2}\Delta r^\eps_{\rm lin} = V^\eps
  u^\eps - \frac{1}{2}\< Q(t)y,y\> v = V^\eps r^\eps_{\rm lin} + 
\(V^\eps - \frac{1}{2}\< Q(t)y,y\>\) v,
\end{equation*}
along with the initial value $r^\eps_{\rm lin\mid t=0}=0$. Since
$V^\eps$ is real-valued, the classical energy estimate for
Schr\"odinger equations yields, in view of \eqref{eq:restepot},
\begin{equation*}
  \lVert r^\eps_{\rm lin}\rVert_{L^\infty([0,t];L^2(\R^d))} \lesssim \sqrt \eps
  \int_0^t \lVert \lvert y\rvert^3v(\tau,y)\rVert_{L^2(\R^d)}d\tau. 
\end{equation*}
Since $Q$ is bounded ($V$ is subquadratic), we have the control
\begin{equation*}
  \lVert \lvert y\rvert^3v(\tau,y)\rVert_{L^2(\R^d)} \le C e^{C\tau}
\end{equation*}
for some constant $C>0$; see Proposition~\ref{prop:v} below. We then
have to notice 
that the wave packet scaling is $L^2$-unitary:
\begin{equation*}
  \|\psi^\eps(t,\cdot)-\varphi^\eps_{\rm
  lin}(t,\cdot)\|_{L^2(\R^d)} = \|u^\eps(t,\cdot)-v(t,\cdot)\|_{L^2(\R^d)}.
\end{equation*}
To summarize, we have:
\begin{lemma}
  Let $d\ge 1$ and $a\in \Sch(\R^d)$. There exists $C>0$ independent
  of $\eps$ such that 
\begin{equation}\label{eq:erreurlin}
  \|\psi^\eps(t,\cdot)-\varphi^\eps_{\rm
  lin}(t,\cdot)\|_{L^2(\R^d)}\le C\sqrt\eps e^{Ct}.
\end{equation}
In particular, there exists $c>0$ independent of $\eps$ such that
\begin{equation*}
 \sup_{0\le t\le c\log\frac{1}{\eps}}\|\psi^\eps(t,\cdot)-\varphi^\eps_{\rm
  lin}(t,\cdot)\|_{L^2(\R^d)} \Tend \eps 0 0 . 
\end{equation*}
\end{lemma}

\subsection{The  nonlinear case}

We now consider the nonlinear situation
$\l\not=0$. Resuming the same change of unknown function
\eqref{eq:chginc}, then adapting the above computation leads to
\begin{equation}\label{eq:ueps}
i\d_t
u^\eps+\frac{1}{2}\Delta u^\eps=V^\eps
u^\eps+\l\eps^{\alpha-\alpha_c} |u^\eps|^{2\sigma}u^\eps,
\end{equation}
where $V^\eps$ is given by \eqref{eq:Veps} as in the linear case, and
\begin{equation}\label{def:alpha}
\alpha_c=1+\frac{d\si}{2}.
\end{equation} 
The real number  $\alpha_c$ appears 
 as a critical exponent. In the case $\alpha>\alpha_c$, we can
 approximate the nonlinear solution $u^\eps$ by the same function $v$
 as in the linear case, given by \eqref{eq:vlinear}. 
The space $\Sigma$ 
will turn out to be quite natural for energy estimates.
Introduce the operators
\begin{equation*}
  A^\eps(t) =  \sqrt \eps \nabla -i\frac{\xi(t)}{\sqrt
    \eps}\quad ;\quad
  B^\eps(t)=\frac{x-x(t)}{\sqrt\eps}.
\end{equation*}
Note that $A$ and $B$ are essentially $\nabla$ and $x$, up to the wave
packet scaling, in the moving frame. From this point of view, our
energy space is quite different from the one associated with the
Lyapounov functional considered in \cite{FGJS04}, and more related to
the one considered in \cite{JFGS06}, since we pay
attention to the localization of the wave packet, through
$B^\eps$. For $f\in \Sigma$, we set  
$$\|f\|_{\mathcal H}=\|f\|_{L^2(\R^d)}+\|A^\eps(t)f\|_{L^2(\R^d)}+\|B^\eps(t)
f\|_{L^2(\R^d)},$$
a notation where we do not emphasize the fact that this norm
depends on $\eps$ and $t$.  

\begin{proposition}\label{linearizablecase}
Let $d\ge 1$, $a\in \Sch(\R^d)$. Suppose  that
$\alpha>\alpha_c$. There exist $C,C_1>0$ independent of
$\eps$, and $\eps_0>0$ such that for all $\eps\in ]0,\eps_0]$,
\begin{equation*}
  \lVert \psi^\eps(t) - \varphi^\eps_{\rm
    lin}(t)\rVert_{\mathcal H}\lesssim\eps^\gamma 
  e^{C_1t},\quad 0\le t\le C\log\frac{1}{\eps},\quad 
\text{where }\gamma=\min\(\frac{1}{2},\alpha-\alpha_c\).
\end{equation*}
In particular, there exists $c>0$ independent of $\eps$ such that 
\begin{equation*}
  \sup_{0\le t\le c\log\frac{1}{\eps}}\lVert \psi^\eps(t) - \varphi^\eps_{\rm
    lin}(t)\rVert_{\mathcal H} \Tend \eps 0 0 . 
\end{equation*}
\end{proposition}
The proof is more complicated than in the linear case (see
\S\ref{sec:linearizable}). The solution of
\eqref{eq:NLS} is linearizable in the sense of \cite{PG96} (see also
\cite{CFG}), up to an Ehrenfest time. 
\smallbreak

In the critical case  $\alpha=\alpha_c$ with $\l\not=0$, the solution
of \eqref{eq:NLS} is no longer linearizable. Indeed, passing formally
to the limit $\eps\to 0$,
Equation~\eqref{eq:ueps} becomes
\begin{equation}\label{eq:u}
i\partial_t u+\frac{1}{2}\Delta u=\frac{1}{2}\langle
Q(t)y,y\rangle u+\l |u|^{2\sigma}u\quad ;\quad u(0,y)=a(y).
\end{equation}
\begin{remark}[Complete integrability]\label{rem:int}
   The cubic one-dimensional case $d=\si=1$ is
special: if $\dot Q=0$, 
then \eqref{eq:u} is completely integrable
(\cite{AblowitzClarkson}). However, if $\dot Q\not =0$, there exists
no Lax pair when the nonlinearity is autonomous as in
\eqref{eq:u}; see \cite{kha09,SHB07}. Note also that if $\dot Q=0$,
then $u^\eps=u$ for all time. 
\end{remark}
As in the linear case, we note that if $V$ is exactly a polynomial of
degree at most two, then $u$ is actually equal to $u^\eps$ for all
$\eps$. The global well-posedness for \eqref{eq:u} has been
established in \cite{Ca-p}. We first prove that $u$ yields a good
approximation for $u^\eps$ on bounded time intervals:
\begin{proposition}\label{prop:bounded}
  Let $d\ge 1$, $\si>0$ with $\si<2/(d-2)$ if $d\ge 3$, and $a\in
  \Sch(\R^d)$. Let $u\in C(\R;\Sigma)$ be the solution to
  \eqref{eq:u}, and let 
\begin{equation}\label{eq:phi}
\varphi^\eps(t,x)=\eps^{-d/4} u
\left(t,\frac{x-x(t)}{\sqrt\eps}\right)e^{i\left(S(t)+\xi(t)\cdot
    (x-x(t))\right)/\eps}.
\end{equation}
For all $T>0$ (independent of $\eps>0$), we have
\begin{equation*}
 \sup_{0\le t\le T} \lVert \psi^\eps(t) -
  \varphi^\eps(t)\rVert_{L^2(\R^d)}=\O\(\sqrt\eps\). 
\end{equation*}
If in addition $\si> 1/2$, 
\begin{equation*}
 \sup_{0\le t\le T} \lVert \psi^\eps(t) -
  \varphi^\eps(t)\rVert_{\mathcal H}=\O\(\sqrt\eps\). 
\end{equation*}
\end{proposition}
\begin{remark}
  The presence of $u$, which solves a nonlinear equation, clearly
  shows that the nonlinearity modifies the coherent state at leading
  order. Note however that the 
Wigner measure of $\psi^\eps$ (see e.g. \cite{GMMP,LionsPaul}) is not
affected by the nonlinearity: 
\begin{align*}
  w(t,x,\xi) &= 
 \|u(t)\|_{L^2(\R^d)}^2 \delta\(x-x(t)\)\otimes \delta\(\xi-\xi(t)\) \\
&= \|a\|_{L^2(\R^d)}^2 
  \delta\(x-x(t)\)\otimes \delta\(\xi-\xi(t)\).
\end{align*}
The Wigner measure remains the same because the nonlinearity alters
only the envelope of the coherent state, not its center in phase space.
\end{remark}
\begin{remark}[Supercritical case] Consider the case
  $\alpha<\alpha_c$, and assume for instance $V=0$. Resuming the
  scaling \eqref{eq:chginc}, Equation~\eqref{eq:ueps} becomes
  \begin{equation*}
    i\d_t u^\eps +\frac{1}{2}\Delta u^\eps = \l
    \eps^{\alpha-\alpha_c}|u^\eps|^{2\si}u^\eps.
  \end{equation*}
  At time $t=0$, $u^\eps$ is independent of $\eps$: $u^\eps_{\mid
    t=0}=a$. Setting $\hbar^2 = \eps^{\alpha_c-\alpha}$ and changing
  the time variable to $s = t/\hbar $, the problem reads
  \begin{equation}\label{eq:semisurcrit}
    i\hbar \d_s u^\hbar +\frac{\hbar^2}{2}\Delta u^\hbar =
    |u^\hbar|^{2\si}u^\hbar \quad ;\quad u^\hbar (0,x)=a(x). 
  \end{equation}
Therefore, to understand the asymptotic behavior of $u$ as $\eps\to 0$
(or equivalently, as $\hbar\to 0$) for $t\in [0,T]$, we need to
understand the large 
time ($s\in [0,T/\hbar]$) behavior in \eqref{eq:semisurcrit}. This
corresponds to a large 
time semi-classical limit in the (supercritical) WKB
regime. Describing this behavior is
extremely delicate, and still an open problem; see  \cite{CaBook}.
\end{remark}
In order to prove the validity of the approximation on large time
intervals, we introduce the following notion:
\begin{definition}
 Let $u\in C(\R;\Sigma)$ be a solution to \eqref{eq:u}, and
 $k\in \N$. We say that $(Exp)_k$ is satisfied if
there exists $C=C(k)$ such that 
  \begin{equation*}
    \forall \alpha, \beta\in \N^d,\ |\alpha|+|\beta|\le k, \quad
    \left\lVert x^\alpha \d_x^\beta u(t)\right\rVert_{L^2(\R^d)}\lesssim
     e^{Ct}. 
  \end{equation*} 
\end{definition} 
Note that reasonably, to establish $(Exp)_k$, the larger the $k$, the
smoother the nonlinearity $z\mapsto |z|^{2\si}z$ has to be. For
simplicity, we shall now assume $\si\in\N$. 
\begin{proposition}[From \cite{Ca-p}]
  Let $d\le 3$, $\si\in \N$ with $\si=1$ if $d=3$, $a\in \Sch(\R^d)$
  and $k\in \N$. Then $(Exp)_k$ is satisfied (at 
  least) in the following cases:
  \begin{itemize}
  \item $\si=d=1$ and $\l\in \R$ (cubic one-dimensional case).
  \item $\si\ge 2/d$, $\l>0$ and $Q(t)$ is diagonal with  eigenvalues
    $\om_j(t)\le 0$. 
\item $\si\ge 2/d$, $\l>0$ and $Q(t)$ is compactly supported.
  \end{itemize}
\end{proposition}
It is very likely that this result remains valid under more general
assumptions (see in particular \cite[\S6.2]{Ca-p} for the case
$\si=2/d$). Yet, we have not been able to prove it. Let us comment a
bit on these three cases. The first case is the most general one
concerning the potential $V$ and the classical trajectory $x(t)$: the
only assumption carries over the nonlinearity (the important aspect is
that it is $L^2$-subcritical). The other two cases concern $L^2$-critical
or supercritical defocusing nonlinearities. In the second case, $V$ is
required to be concave (along the classical trajectory), and the last
case corresponds for instance to a compactly supported ${\rm Hess V}$,
when the classical trajectory is not trapped. In
this last case, we have actually better than an exponential decay:
Sobolev norms are bounded, and momenta grow algebraically in
time. The following result could be improved in this case. 

\begin{theorem}\label{theo:consistency}
Let $a\in \Sch(\R^d)$. If $(Exp)_4$ is satisfied, then
there exist  $C,C_2>0$ independent of $\eps$, and $\eps_0>0$ such that
for all $\eps\in ]0,\eps_0]$,
\begin{equation*}
  \lVert \psi^\eps(t) -
  \varphi^\eps(t)\rVert_{\mathcal H}\lesssim\sqrt \eps \exp\(\exp(C_2t)\),\quad
  0\le t\le C \log\log\frac{1}{\eps}.
\end{equation*}
In particular, there exists $c>0$ independent of $\eps$ such that 
\begin{equation*}
  \sup_{0\le t\le c\log\log\frac{1}{\eps}}\lVert \psi^\eps(t) -
  \varphi^\eps(t)\rVert_{\mathcal H} \Tend \eps 0 0 .  
\end{equation*}
\end{theorem}
In the one-dimensional cubic case, this result can be improved on two
aspects. First, we can prove a long time asymptotics in $L^2$ provided
$(Exp)_3$ is satisfied. More important is the fact that we obtain an
asymptotics up to an Ehrenfest time:
\begin{theorem}\label{theo:cubic}
  Assume $d=\si=1$, and let $a\in \Sch(\R)$. If $(Exp)_3$ is satisfied, then
there exist  $C,C_3>0$ independent of $\eps$, and $\eps_0>0$ such that
for all  $\eps\in ]0,\eps_0]$,
\begin{equation*}
  \lVert \psi^\eps(t) -
  \varphi^\eps(t)\rVert_{L^2(\R)}\lesssim\sqrt \eps \exp(C_3t),\quad
  0\le t\le C \log\frac{1}{\eps}.
\end{equation*}
In particular, there exists $c>0$ independent of $\eps$ such that 
\begin{equation*}
  \sup_{0\le t\le c\log\frac{1}{\eps}}\lVert \psi^\eps(t) -
  \varphi^\eps(t)\rVert_{L^2(\R)} \Tend \eps 0 0 .  
\end{equation*}
If in addition $(Exp)_4$ is satisfied, then for the same constants as
above,
\begin{equation*}
  \lVert \psi^\eps(t) -
  \varphi^\eps(t)\rVert_{\mathcal H}\lesssim\sqrt \eps \exp(C_3t),\quad
  0\le t\le C \log\frac{1}{\eps},
\end{equation*}
and
\begin{equation*}
  \sup_{0\le t\le c\log\frac{1}{\eps}}\lVert \psi^\eps(t) -
  \varphi^\eps(t)\rVert_{\mathcal H} \Tend \eps 0 0 .  
\end{equation*}
\end{theorem}
The technical reason which explains the differences between
Theorem~\ref{theo:consistency} and Theorem~\ref{theo:cubic} is that
the one-dimensional cubic case is $L^2$-subcritical. This aspect has
several consequences regarding the Strichartz estimates we use in the
course of the proof.
\smallbreak

These nonlinear results are to be compared with previous ones
concerning the interaction between a linear dynamics (classical
trajectories) and nonlinear effects. Consider the WKB regime
\begin{equation}\label{eq:BKW2}
  i\eps\d_t \psi^\eps+\frac{\eps^2}{2}\Delta \psi^\eps = V(x)\psi^\eps +
  \l |\psi^\eps|^{2\si} \psi^\eps\quad ;\quad \psi^\eps(0,x)= \eps^{\widetilde
    \beta} a(x) e^{ix\cdot \xi_0/\eps}, 
\end{equation}
with $V$ satisfying Assumption~\ref{hyp:V}. Like above, it is
equivalent, up to a rescaling, to
\begin{equation*}
  i\eps\d_t \psi^\eps+\frac{\eps^2}{2}\Delta \psi^\eps = V(x)\psi^\eps +
  \l \eps^{\widetilde \alpha}|\psi^\eps|^{2\si} \psi^\eps\quad ;\quad
  \psi^\eps(0,x)=  a(x) e^{ix\cdot \xi_0/\eps},
\end{equation*}
with $\widetilde \alpha =2\si\widetilde \beta$. The critical value in
this regime is $\widetilde 
\alpha_c=1$ (see \cite{CaBook}). In \eqref{eq:BKW2}, this corresponds
to initial data of 
order $\eps^{1/(2\si)}$ in $L^\infty$, like in the present case of
wave packets. However, the critical nonlinear effects are very
different in the case of \eqref{eq:BKW2}. The following asymptotics
holds in $L^2(\R^d)$ (see \cite{CaBook}): 
\begin{equation*}
  \psi^\eps(t,x) \Eq \eps 0 a(t,x)e^{ig(t,x)}e^{i\phi(t,x)/\eps},
\end{equation*}
as long as the phase $\phi$, solution to the Hamilton--Jacobi equation
\begin{equation*}
  \d_t \phi+\frac{1}{2}|\nabla \phi|^2 +V=0\quad ;\quad
  \phi(0,x)=x\cdot \xi_0,
\end{equation*}
remains smooth.  More general initial phases are actually allowed: we
consider an initial phase linear in $x$ for the comparison with
\eqref{eq:NLS}. The amplitude $a$ solves a \emph{linear} transport
equation: at leading order, nonlinear effects show up through the
phase modulation $g$ (which depends on $\l$ and $\si$). This result
calls for at least two comments. First, this nonlinear effect is rather
weak: for instance, it does not affect the main quadratic observables
at leading order,
$|\psi^\eps|^2$ (position density) and $\eps\IM \overline \psi^\eps
\nabla \psi^\eps$ (current density). In the case of \eqref{eq:NLS},
the profile equation is, in a sense, more nonlinear, even though in
both cases, Wigner measures are not affected by the critical
nonlinearity. Second, the validity 
of WKB analysis is limited in general, even if $V$ is a polynomial. If
$V=0$, $\phi(t,x) = x\cdot \xi_0 -t|\xi_0|^2/2$ is smooth for all
time, $a(t,x)=a_0(x-t\xi_0)$ remains bounded, and the asymptotics can
be justified up to Ehrenfest time, by simply resuming the proof given
in \cite{CaBook}. If $V(x)=E\cdot x$, Avron--Herbst formula shows
that this case is essentially the same as $V=0$. On the other hand, if
$V(x)=\om^2|x|^2/2$, classical trajectories in \eqref{eq:traj} are
explicit:
\begin{equation*}
  x(t)=x_0\cos(\om t)+\xi_0\frac{\sin(\om t)}{\om}. 
\end{equation*}
They all meet at $\xi_0/\om$ at time $t_*=\pi/(2\om)$: the phase $\phi$
becomes singular as $t\to t_*$, and WKB analysis ceases to be valid,
while the wave packets approach yields an exact result for all time in
such a case. 
\smallbreak

In
\cite{BJ00,FGJS04,GS07,HZ07,JFGS06,Keraani02,KeraaniAA}, the authors
have considered a similar problem, in a different regime though:
\begin{equation}\label{eq:WKB}
  i\eps\d_t \psi^\eps+\frac{\eps^2}{2}\Delta \psi^\eps = V(x)\psi^\eps -
  |\psi^\eps|^{2\si} \psi^\eps\quad ;\quad \psi^\eps(0,x)=Q
\(\frac{x-x_0}{\eps}\)e^{i\xi_0\cdot x/\eps}, 
\end{equation}
where $Q$ is a ground state, solution to a nonlinear elliptic
equation. 
They prove, with some precision depending on the papers:
\begin{equation*}
  \psi^\eps(t,x) \Eq \eps 0 Q\(\frac{x-x(t)}{\eps}\)e^{i\xi(t)\cdot
    x/\eps + i\theta^\eps(t)},\quad \theta^\eps(t)\in \R. 
\end{equation*}
As pointed out in \cite{HZ07}, such results may be extended to an
Ehrenfest time. An important difference with our paper must be
emphasized, besides the scaling: the particular initial data
makes it possible to rely on rigidity properties of the solitary
waves, which do not hold for general profiles. In \cite{CM04},
some results concerning a defocusing equation with more
general initial profiles are proved (or cited), in the same scaling as in
\eqref{eq:WKB}: however, it seems that unless $V$ is a polynomial of
degree at most two, only partial results are available then (that is,
on relatively small time intervals). Finally, even when $\d^\gamma V=$ for all
$|\gamma|\ge 3$, the time intervals on which some asymptotic results
are proved must be independent of $\eps$.   
\smallbreak

\subsection{Nonlinear superposition}

We still suppose $\alpha=\alpha_c$. For simplicity, in this paragraph,
we assume that $\si$ is an integer: this is compatible with the fact
that the nonlinearity is energy-subcritical only if $d\le 3$. We
consider initial data corresponding to the superposition of two  wave
packets:
\begin{equation*}
 \psi^\eps(0,x)=\eps^{-d/4} a_1\left(\frac{x-x_1}{\sqrt\eps}\right)
e^{i{(x-x_1)\cdot\xi_1/\eps}}+\eps^{-d/4} a_2\left(\frac{x-x_2}{\sqrt\eps}\right)
e^{i{(x-x_2)\cdot\xi_2/\eps}}, 
\end{equation*}
with $a_1,a_2\in \Sch({\R})$,
$(x_1,\xi_1),(x_2,\xi_2)\in\R^{2}$, and $(x_1,\xi_1)\not
=(x_2,\xi_2)$. For $j\in\{1,2\}$,  $\left(x_j(t),\xi_j(t)\right)$ are the 
classical trajectories solutions to
\eqref{eq:traj} with initial data $(x_j,\xi_j)$. We denote by~$S_j$
the action associated with $(x_j(t),\xi_j(t))$ by~\eqref{eq:action} and
by $u_j$ the solution of~\eqref{eq:u} for the curve $x_j(t)$ and with
initial data $a_j$. We consider $ 
\varphi_j^\eps$  associated by~\eqref{eq:phi} with $u_j, x_j,\xi_j,
S_j $, and $\psi^\eps \in C(\R;\Sigma)$ solution to \eqref{eq:NLS} with
$\alpha=\alpha_c$ and the above initial data.
\smallbreak

The functional setting used to describe the function $\psi^\eps$ must
be changed in the case of two initial wave packets: recall that $\mathcal
H$ is defined through $A^\eps$ and $B^\eps$, which are related to the
Hamiltonian flow. The geometric meaning of $A^\eps$ and $B^\eps$
becomes irrelevant in the case of two wave packets. Instead, we use
norms on $\Sigma$ whose geometric meaning is weaker, since
essentially, they reflect the fact that we consider $\eps$-oscillatory
functions, which remain somehow localized in space (before Ehrenfest
time):
\begin{equation*}
  \lVert f\rVert_{\Sigma_\eps} = \lVert f\rVert_{L^2(\R^d)} + \lVert
  \eps \nabla f\rVert_{L^2(\R^d)} +\lVert x f\rVert_{L^2(\R^d)} .
\end{equation*}
For finite time, we have: 
\begin{proposition}\label{prop:superposition}
Let $d\le 3$,  $\si\in \N$ ($\si=1$ if $d=3$), and $a_1,a_2\in
\Sch(\R^d)$.  For
all $T>0$ (independent  of $\eps$), we have, for all 
$\gamma <1/2$:
\begin{equation*}
  \sup_{0\le t\le
    T}\|\psi^\eps(t)-\varphi_1^\eps(t)-\varphi_2^\eps(t)\|_{\Sigma_\eps}=
  \O\(\eps^{\gamma}\).  
\end{equation*}
\end{proposition}

Besides, nonlinear superposition holds for large time (at least) in
the one-dimensional case, if
the points $\(x_1,\xi_1\)$ and $\(x_2,\xi_2\)$ have different energies. 
 \begin{theorem}\label{theo:superposition}
Assume that $d=1$, $\si$ is an integer, and let $a_1,a_2\in \Sch(\R)$.
Suppose that $E_1\not=E_2$, where 
\begin{equation*}
  E_j = \frac{\xi_j^2}{2}+V\(x_j\). 
\end{equation*}
Suppose that  $(Exp)_{k}$ is satisfied for some $k\ge 4$ (for $u_1$
and $u_2$).\\
$1.$ There exist $C,C_3>0$ independent of $\eps$, and $\eps_0>0$  such that for
all $\eps\in ]0,\eps_0]$,
\begin{equation*}
  \|\psi^\eps(t)-\varphi_1(t)^\eps-\varphi_2^\eps(t)\|_{\Sigma_\eps}
  \lesssim \eps^{\gamma}\exp\(\exp (C_3 t)\),\quad 0\le t\le
  C\log\log\frac{1}{\eps},
\end{equation*}
with $\gamma=\frac{k-2}{2k-2}$. In particular, there exists $c>0$
independent of $\eps$ such that 
\begin{equation*}
  \sup_{0\le t\le c\log\log\frac{1}{\eps}}
  \|\psi^\eps(t)-\varphi_1(t)^\eps-\varphi_2^\eps(t)\|_{\Sigma_\eps}
  \Tend \eps 0 0 . 
\end{equation*}
$2.$ Suppose in addition that $\si=1$.
There exist $C,C_4>0$ independent of $\eps$, and $\eps_0>0$  such that for
all $\eps\in ]0,\eps_0]$,
\begin{equation*}
  \|\psi^\eps(t)-\varphi_1(t)^\eps-\varphi_2^\eps(t)\|_{\Sigma_\eps}
  \lesssim \eps^{\gamma}e^{C_4 t},\quad 0\le t\le
  C\log\frac{1}{\eps},\quad \text{with }\gamma=\frac{k-2}{2k-2} . 
\end{equation*}
In particular, there exists $c>0$ independent of $\eps$ such that
\begin{equation*}
  \sup_{0\le t\le c\log\frac{1}{\eps}}
  \|\psi^\eps(t)-\varphi_1(t)^\eps-\varphi_2^\eps(t)\|_{\Sigma_\eps}
  \Tend \eps 0 0 . 
\end{equation*}
\end{theorem}

It is interesting to see that even though the profiles are nonlinear,
the superposition principle, which is a property of linear equations,
still holds. There are many other such nonlinear superposition
principles in the literature, and we cannot mention them all.
\smallbreak

This result is to be compared with those in
\cite{KMR09} (see also references therein), for several reasons. In
\cite{KMR09}, the authors 
construct a solution for the three-dimensional Schr\"odinger--Poisson
system which behaves, in $H^1(\R^3)$ and asymptotically for large
time, like the sum of 
two ground state solitary waves. The two solitary
waves are centered, in the phase space, at the solution of a two-body
problem: unlike what happens in our case, there exists an interaction
between the trajectories, due to the fact that the Poisson potential
is long range. In our case, the long range aspect of the nonlinearity
(when $d=\si=1$; see \cite{Ozawa91}) does not have such a consequence:
we will see 
that the key point in the proof of the above two results is the fact
that in the wave packet scaling, the two functions $\varphi_1^\eps$
and $\varphi_2^\eps$ do not interact at leading order in the limit
$\eps\to 0$: the nonlinear effects concentrate on the profiles, along
the classical trajectories, and it turns out that  these trajectories
do not meet ``too much''. In \cite{CFG}, another nonlinear
superposition principle was proved, in the scaling of
\eqref{eq:WKB}. However, in \cite{CFG}, nonlinear effects were
localized in space \emph{and} time, so most of the time, the nonlinear
superposition was actually a linear one.

\subsection{Outline of the paper}

In Section~\ref{sec:linearizable},
we first analyze the linearizable  case and prove
Proposition~\ref{linearizablecase} after a short analysis of the
linear case.  Then, in 
Section~\ref{sec:bounded}, we recall basic facts about Strichartz
estimates in this semi-classical framework and prove the
consistency of our approximation on bounded time intervals.
Theorem~\ref{theo:consistency} is proved in
Section~\ref{sec:consistency}. Finally, Section~\ref{sec:cubic} is
focused on the one-dimensional cubic case and  
Section~\ref{sec:superposition} on the analysis of the
nonlinear superposition.  

\begin{notation}
Throughout the paper, in the expression $e^{Ct}$, the constant
$C$ will denote a constant independent of $t$ which may change from
one line to the other. 
\end{notation}

\section{The linearizable case}
\label{sec:linearizable}

\smallbreak

In this section, we assume $\alpha>\alpha_c$ and we  prove
Proposition~\ref{linearizablecase}. We 
first recall estimates in the linear case $\l=0$ which are more 
precise than in \S\ref{sec:lin1}.

\subsection{The linear case}\label{sec:linear}

We suppose here $\lambda=0$. The first remark concerns the
properties of the profile $v$. It is not difficult to prove the
following proposition. 

\begin{proposition}\label{prop:v}
Let $d\ge 1$ and $a\in \Sch(\R^d)$. 
For all $k\in\N$,
there exists $C>0$ such that the solution $v$ to \eqref{eq:vlinear}
satisfies  
\begin{equation*}
  \forall \alpha,\beta\in\N^d, \;\;|\alpha|+|\beta|\le
  k,\;\;\left\|x^\alpha\partial_x^\beta
    v(t)\right\|_{L^2(\R^d)}\lesssim e^{Ct}. 
\end{equation*}
\end{proposition}
A general proof of Proposition~\ref{prop:v} is given for instance in
\cite[\S6.1]{Ca-p}. 
Let us now consider $w^\eps=
\psi^\eps-\varphi^\eps_{\rm lin}$. We have $w^\eps(0)=0$ and  
\begin{equation*}
    i\eps\partial_t
w^\eps+\frac{\eps^2}{2}\Delta w^\eps=V(x)w^\eps-\left(V(x)-
T_2(x,x(t))\right)\varphi^\eps_{\rm lin},
\end{equation*}
where $T_2$ corresponds to a second order Taylor approximation:
\begin{equation*}
T_2(x,a)  :=  V(a)+\<\nabla
V(a),x-a\>+\frac{1}{2}\<\textrm{Hess}
V(a)(x-a),x-a\>. 
\end{equation*}
We have seen in \S\ref{sec:lin1} that 
the standard $L^2$ estimate for Schr\"odinger equations yields
$$\|w^\eps(t)\|_{L^2(\R^d)}\lesssim \sqrt\eps \|y^3
v(t)\|_{L^2(\R^d)}\lesssim \sqrt\eps e^{Ct}.$$ 
 In order to analyze the convergence in $\Sigma$, we can write
\begin{align*}
\left(i\eps\partial_t+\frac{\eps^2}{2} \Delta
  -V(x)\right)\left(\eps\nabla w^\eps\right) & =  
\eps \nabla  V w ^\eps -\eps \nabla L^\eps, \\
\left(i\eps\partial_t+\frac{\eps^2}{2} \Delta
  -V(x)\right)\left(xw^\eps\right) & =
\left[\frac{\eps^2}{2}\Delta,x\right] w^\eps -x L^\eps = \eps^2 \nabla
w^\eps -xL^\eps, 
\end{align*}
where 
\begin{equation}\label{def:Leps}
L^\eps(t,x)  :=   \left(V(x)-
T_2(x,x(t))\right)\varphi^\eps_{\rm lin}(t,x).
\end{equation}
Typically if $d=1$,
\begin{align*}
 L^\eps(t,x) &= \frac{1}{2}\(x-x(t)\)^3 \varphi^\eps_{\rm
   lin}(t,x)\int_0^1 V^{'''}\(x(t)+\theta\(x-x(t)\)\)\theta^2d\theta\\
&=  \frac{\(x-x(t)\)^3 }{2\eps^{1/4}}
 e^{-i\left(S(t)+\xi(t)\cdot
    (x-x(t))\right)/\eps}v
\left(t,\frac{x-x(t)}{\sqrt\eps}\right)
I\(x,x(t)\),
\end{align*}
where
$$I\(x,x(t)\)=\int_0^1 V^{'''}\(x(t)+\theta\(x-x(t)\)\)\theta^2d\theta.$$
Energy estimates make it possible to show
$$\|\eps\nabla
w^\eps(t)\|_{L^2(\R^d)}+\|xw^\eps(t)\|_{L^2(\R^d)}\lesssim \sqrt \eps
e^{Ct}.$$ 
However, the operators $A^\eps$ and $B^\eps$ defined in the introduction yield
more precise results. For instance, $\|\eps \nabla \varphi_{\rm lin}\|_{L^2}$ is
of order $\O(1)$ exactly, because of the phase factor in
\eqref{eq:philin}. We note the formula
\begin{equation}\label{eq:A}
  A^\eps(t) = \sqrt \eps \nabla -i\frac{\xi(t)}{\sqrt
    \eps}= \sqrt \eps e^{i(S(t)+\xi(t)\cdot
    (x-x(t)))/\eps} \nabla \(e^{-i(S(t)+\xi(t)\cdot
    (x-x(t)))/\eps} \cdot\) , 
\end{equation}
so for instance $\|A^\eps(t)\varphi_{\rm lin}\|_{L^2}$ is of order $\O(1)$:
morally, we have gained a factor $\sqrt\eps$. 
\begin{lemma}\label{lem:AB}
  The operators $A^\eps$ and $B^\eps$, defined by
  \begin{equation*}
   A^\eps(t) = \sqrt \eps \nabla -i\frac{\xi(t)}{\sqrt
    \eps} \quad ;\quad B^\eps(t)=\frac{x-x(t)}{\sqrt\eps},  
  \end{equation*}
satisfy the commutation relations:
\begin{align*}
  \left[i\eps\d_t +\frac{\eps^2}{2}\Delta -V,A^\eps(t)\right] &=\sqrt \eps
  \(\nabla V(x)-\nabla V\(x(t)\)\),\\
 \left[i\eps\d_t +\frac{\eps^2}{2}\Delta -V,B^\eps(t)\right] &=\eps A^\eps(t).
\end{align*}
\end{lemma}
We can then write
\begin{align*}
\left(i\eps\partial_t+\frac{\eps^2}{2} \Delta
  -V(x)\right)A^\eps(t) w^\eps & =  
\sqrt\eps \(\nabla  V(x)-\nabla V\(x(t)\)\) w ^\eps -A^\eps(t) L^\eps, \\
\left(i\eps\partial_t+\frac{\eps^2}{2} \Delta
  -V(x)\right)B^\eps(t)w^\eps & =    \eps A^\eps(t) w^\eps -B^\eps(t)L^\eps.
\end{align*}
In view of \eqref{eq:A}, we observe
\begin{align*}
  \lVert A^\eps(t) L^\eps\rVert_{L^2(\R^d)}&\lesssim \eps^{3/2}\|x^2
  v(t)\|_{L^2(\R^d)} +\eps^{3/2}\|x^3 \nabla v(t)\|_{L^2(\R^d)} + \eps^2 \|x^3
  v(t)\|_{L^2(\R^d)}\\
&\lesssim \eps^{3/2}e^{Ct},
\end{align*}
thanks to Lemma~\ref{prop:v}. Similarly,
\begin{equation*}
  \lVert B^\eps(t) L^\eps\rVert_{L^2(\R^d)}\lesssim \eps^{3/2}e^{Ct}.
\end{equation*}
Since we have the pointwise estimate
\begin{equation*}
  \left\lvert \sqrt\eps \(\nabla  V(x)-\nabla V\(x(t)\)\)
    w^\eps\right\rvert 
\lesssim \eps\left\lvert B^\eps(t) w^\eps\right\rvert, 
\end{equation*}
energy estimates yield
\begin{equation*}
  \|w^\eps(t)\|_{\mathcal H} 
  \lesssim \int_0^t \(
  \|w^\eps(s)\|_{\mathcal H} +\sqrt\eps e^{Cs}\)ds.
\end{equation*}
We conclude by Gronwall Lemma:
\begin{equation*}
 \|w^\eps(t)\|_{\mathcal H}  \lesssim \sqrt\eps e^{Ct}.
\end{equation*}
We will see in the following subsection that the arguments are somehow
more complicated in the nonlinear setting.

\subsection{Proof of Proposition~\ref{linearizablecase}}

We now assume $\l\not =0$, and $\alpha>\alpha_c$. For
the simplicity of the presentation, we give the detailed proof in the
case $d=1$ only.

We set again $w^\eps=\psi^\eps-\varphi^\eps_{\rm lin}$ and we write
the equation satisfied by $w^\eps$:
\begin{equation*}
    i\eps\partial_t
w^\eps+\frac{\eps^2}{2}\d_x^2 w^\eps=V(x)w^\eps-\left(V(x)-
T_2(x,x(t))\right)\varphi^\eps_{\rm lin}+N^\eps\quad ;\quad
w^\eps_{\mid t=0}=0,
\end{equation*}
where the nonlinear source term is given by
\begin{equation*}
 N^\eps= \l\eps^{\alpha} |\varphi^\eps_{\rm
  lin}+w^\eps|^{2\sigma} (\varphi^\eps_{\rm lin}+w^\eps). 
\end{equation*}
First, since $\l\eps^{\alpha} |\varphi^\eps_{\rm
  lin}+w^\eps|^{2\sigma}\in \R$, the $L^2$ energy estimate for
$w^\eps$ yields
\begin{equation*}
   \|w^\eps(t)\|_{L^2(\R)}\lesssim \frac{1}{\eps}\|L^\eps\|_{L^1([0,t];L^2(\R))}+
\frac{1}{\eps}\left\lVert\eps^{\alpha} |\varphi^\eps_{\rm
  lin}+w^\eps|^{2\sigma} \varphi^\eps_{\rm lin}\right\rVert_{L^1([0,t];L^2(\R))},
\end{equation*}
where we have kept the notation \eqref{def:Leps}.
The contribution of $N^\eps$ cannot be studied directly, since we do
not know yet how to estimate $w^\eps$: since $w^\eps$ will turn out to
be small, we use a 
 bootstrap argument. 

Since we have
 $\|\varphi^\eps_{\rm
   lin}(t)\|_{L^\infty(\R)} = \eps^{-1/4} \|v(t)\|_{L^\infty(\R)}$, 
Proposition~\ref{prop:v} and Sobolev embedding show that there exists
$C_0>0$ such that
\begin{equation*}
  \|\varphi^\eps_{\rm   lin}(t)\|_{L^\infty(\R)}\le C_0 \eps^{-1/4}
  e^{C_0t},\quad \forall t\ge 0. 
\end{equation*}
The bootstrap argument goes as follows. We suppose that
 for $t\in [0,\tau]$ we have  
 \begin{equation}
   \label{eq:solong1}
   \|w^\eps(t)\|_{L^\infty}\le \eps^{-1/4}e^{C_0t},
 \end{equation}
with the same constant $C_0$. 
Since $w^\eps_{\mid t=0}=0$ and $\psi^\eps\in C(\R;\Sigma)$, there
exists $\tau^\eps>0$ (\emph{a priori} depending on $\eps$) such that
\eqref{eq:solong1} holds on $[0,\tau^\eps]$. So long as
\eqref{eq:solong1} holds,
\begin{equation*}
 \left\lVert\eps^{\alpha} |\varphi^\eps_{\rm
  lin}+w^\eps|^{2\sigma} \varphi^\eps_{\rm lin}\right\rVert_{L^2(\R)}
\lesssim \eps^{\alpha-\sigma/2}\|a\|_{L^2(\R)}e^{2\si C_0t}.
\end{equation*}
We infer
\begin{equation*}
  \|w^\eps(t)\|_{L^2(\R)}\lesssim \sqrt\eps e^{Ct}+
  \eps^{\alpha-\alpha_c}e^{2\si C_0t}. 
\end{equation*}
Applying the operators $A^\eps$ and $B^\eps$ to the equation
satisfied by $w^\eps$, we find:
\begin{align*}
\left(i\eps\partial_t+\frac{\eps^2}{2} \d_x^2
  -V(x)\right)A^\eps w^\eps & =  
\sqrt\eps \(  V'(x)- V'\(x(t)\)\) w ^\eps  -A^\eps
L^\eps+A^\eps N^\eps , \\
\left(i\eps\partial_t+\frac{\eps^2}{2} \d_x^2
  -V(x)\right)B^\eps w^\eps & =    \eps A^\eps w^\eps
-B^\eps L^\eps+B^\eps N^\eps.
\end{align*}
We observe that in view of \eqref{eq:A}, $A^\eps$ acts on gauge
invariant non linearities like a derivative. Therefore, so long as
\eqref{eq:solong1} holds, 
\begin{align*}
  \left\| A^\eps(t)N^\eps(t)\right\|_{L^2(\R)}
 &\lesssim \eps^{\alpha}
 \(\|\varphi^\eps_{\rm lin}(t)\|_{L^\infty(\R)}^{2\sigma}+
\|w^\eps(t)\|_{L^\infty(\R)}^{2\sigma}\) \|A^\eps(t)
\varphi^\eps_{\rm lin}(t)\|_{L^2(\R)} \\
&+ \eps^{\alpha}
 \(\|\varphi^\eps_{\rm lin}(t)\|_{L^\infty(\R)}^{2\sigma}+
\|w^\eps(t)\|_{L^\infty(\R)}^{2\sigma}\)  \|A^\eps(t) w^\eps(t)\|_{L^2(\R)}\\
&\lesssim \eps^{\alpha-\sigma/2} e^{2\si C_0 t}\(e^{Ct}+ \|A^\eps(t)
w^\eps(t)\|_{L^2(\R)}\). 
\end{align*}
Similarly, we obtain
\begin{align*}
  \left\| B^\eps(t)N^\eps(t)\right\|_{L^2(\R)}& \lesssim
  \eps^{\alpha-\sigma/2} e^{2\si C_0 t}\(e^{Ct}+ \|B^\eps(t)
w^\eps(t)\|_{L^2(\R)}\).
\end{align*}
We infer, thanks to the linear estimates,
\begin{align*}
   \|A^\eps(t) w^\eps(t)\|_{L^2(\R)}&\lesssim
  \|B^\eps w^\eps\|_{L^1([0,t];L^2(\R))}  + \sqrt \eps e^{Ct}\\&
+
  \eps^{\alpha-\alpha_c} \int_0^t e^{2\si C_0 s}\(e^{Cs}+ \|A^\eps(s)
w^\eps(s)\|_{L^2(\R)}\)ds,\\
\|B^\eps(t) w^\eps(t)\|_{L^2(\R)}&\lesssim \|A^\eps 
  w^\eps\|_{L^1([0,t];L^2(\R))} + \sqrt \eps e^{Ct}\\
& +  \eps^{\alpha-\alpha_c}\int_0^t e^{2\si C_0 s}\(e^{Cs}+ \|B^\eps(s)
w^\eps(s)\|_{L^2(\R)}\)ds .
\end{align*}
Gronwall Lemma yields, so long as \eqref{eq:solong1} holds:
\begin{align*}
\|w^\eps(t)\|_{\mathcal H}&\lesssim  \int_0^t  
  \eps^\gamma e^{Cs} \exp\( C
  \eps^{\alpha-\alpha_c}\int_s^t e^{2\si C_0 s'}ds'\)ds\\
&\lesssim  \exp\( C
  \eps^{\alpha-\alpha_c}e^{2\si C_0 t}\)\int_0^t 
  \eps^\gamma e^{Cs} ds\lesssim  \exp\( C
  \eps^{\alpha-\alpha_c}e^{2\si C_0 t}\) 
  \eps^\gamma e^{Ct},
\end{align*}
where $\gamma = \min(1/2,\alpha-\alpha_c)$. First, we notice that 
\begin{equation*}
  \eps^{\alpha-\alpha_c}e^{2\si C_0 t}\le 1\quad \text{for }0\le t\le
  \frac{\alpha-\alpha_c}{2\si C_0} \log\frac{1}{\eps}. 
\end{equation*}
Then, setting $\kappa =\frac{\alpha-\alpha_c}{2\si C_0}$,
Gagliardo-Nirenberg inequality yields, so long as \eqref{eq:solong1}
holds, with also $t\le \kappa \log\frac{1}{\eps}$, and thanks to the
factorization \eqref{eq:A}, 
\begin{align*}
  \|w^\eps(t)\|_{L^\infty(\R)} \lesssim \frac{1}{
    \eps^{1/4}}\|w^\eps(t)\|^{1/ 2}_{L^2(\R)} \|A^\eps(t)
w^\eps(t)\|_{L^2(\R)}^{1/ 2}
\lesssim \eps^{\gamma-1/4} e^{Ct}. 
\end{align*}
This is enough to show that the bootstrap argument \eqref{eq:solong1}
works provided the time variable is restricted to 
\begin{equation*}
  C\eps^\gamma e^{Ct}\le 1,
\end{equation*}
that is, $0\le t\le C\log\frac{1}{\eps}$ for some $C>0$ independent of
$\eps$. Proposition~\ref{linearizablecase} follows in the case $d=1$.

\smallbreak

To prove Proposition~\ref{linearizablecase} when $d\ge 2$, one can use
Strichartz estimates. This approach 
is more technical. Since the case
$\alpha>\alpha_c$ does not seem the most interesting one, and since we
will use Strichartz estimates in the fully nonlinear case, we choose
not to present the proof of Proposition~\ref{linearizablecase} when
$d\ge 2$.

\section{Fully nonlinear case: bounded time intervals}
\label{sec:bounded}

In this section, we prove Proposition~\ref{prop:bounded}. This gives
us the opportunity to introduce some technical tools which will be
used to study large time asymptotics. 

\subsection{ Strichartz estimates}
\label{sec:strichartz}
\begin{definition}\label{def:adm}
 A pair $(q,r)$ is {\bf admissible} if $2\le r
  \le\frac{2d}{d-2}$ (resp. $2\le r\le \infty$ if $d=1$, $2\le r<
  \infty$ if $d=2$)  
  and 
$$\frac{2}{q}=\delta(r):= d\left( \frac{1}{2}-\frac{1}{r}\right).$$
\end{definition}
Following \cite{GV85,Yajima87,KT}, Strichartz estimates are available
for the Schr\"odinger equation without external potential. Thanks to
the construction of the parametrix performed in
\cite{Fujiwara79,Fujiwara}, similar results are available in the
presence of an external satisfying Assumption~\ref{hyp:V} ($V$ could
even depend on time).
Denote by $U^\eps(t)$ the semi-group associated to
$-\frac{\eps^2}{2}\Delta +V$: $\phi^\eps(t,x)=U^\eps(t)\phi_0(x)$ if
\begin{equation*}
  i\eps \partial_t \phi^\eps +\frac{\eps^2}{2}\Delta \phi^\eps = V \phi^\eps\quad
  ;\quad \phi^\eps(0,x)=\phi_0(x).  
\end{equation*}
From \cite{Fujiwara79}, it satisfies the following properties:
\begin{itemize}
\item The map $t\mapsto U^\eps(t)$ is strongly continuous.
\item $U^\eps(t)U^\eps(s)=U^\eps(t+s)$. 
\item $U^\eps(t)^*=U^\eps(t)^{-1}=U^\eps(-t)$.
\item $U^\eps(t)$ is unitary on $L^2$:
  $\|U^\eps(t)\phi\|_{L^2(\R^d)}= \|\phi\|_{L^2(\R^d)}$.
\item Dispersive properties: there exist $\delta,C>0$
  independent of $\eps\in ]0,1]$ such
  that for all $ t\in \R$ with $|t|\le \delta$,
  \begin{equation*}
    \|U^\eps(t)\|_{L^1(\R^d)\to L^\infty(\R^d)}\le \frac{C}{(\eps|t|)^{d/2}}. 
  \end{equation*}
\end{itemize}
We infer the following result, from \cite{KT}:
\begin{lemma}[Scaled Strichartz
  inequalities]
\label{lem:scaledStrichartz}
Let $(q,r)$, $(q_1,r_1)$ and~$ (q_2,r_2)$ be admissible pairs. Let $I$
be some finite time interval. \\
$1.$ There exists $C=C(r,|I|)$ independent of $\eps$, such that for
all $\phi \in 
L^2(\R^d)$,
\begin{equation}\label{eq:strich}
    \eps^{1/q} \left\| U^\eps(\cdot)\phi \right\|_{L^q(I;L^r(\R^d))}\le C
    \|\phi 
    \|_{L^2(\R^d)}.
  \end{equation}
$2.$ If $I$ contains the origin, $0\in I$, denote
\begin{equation*}
  D_I^\eps(F)(t,x) = \int_{I\cap\{s\le
      t\}} U^\eps(t-s)F(s,x)ds. 
\end{equation*}
There exists $C=C(r_1,r_2,|I|)$ independent of $\eps$
such that  for all $F\in L^{q'_2}(I;L^{r'_2})$,
\begin{equation}\label{eq:strichnl}
      \eps^{1/q_1+1/q_2}\left\lVert D_I^\eps(F)
      \right\rVert_{L^{q_1}(I;L^{r_1}(\R^d))}\le C \left\lVert
      F\right\rVert_{L^{q'_2}\(I;L^{r'_2}(\R^d)\)}.
\end{equation}
\end{lemma}

\subsection{Proof of Proposition~\ref{prop:bounded}}
\label{sec:proofbounded}
Denote the error term by
$w^\eps=\psi^\eps-\varphi^\eps$, where $\varphi^\eps$ is now given by
\eqref{eq:phi}, and $u\in C(\R;\Sigma)$ satisfies \eqref{eq:u}. This remainder
solves
\begin{equation}\label{eq:restecrit}
  i\eps\d_t w^\eps +\frac{\eps^2}{2}\Delta w^\eps =V w^\eps -\mathcal L^\eps
  + 
  \l\eps^{\alpha_c}\(|\psi^\eps|^{2\si}\psi^\eps
  -|\varphi^\eps|^{2\si}\varphi^\eps\) \quad ;\quad w^\eps_{\mid t=0}=0, 
\end{equation}
where 
\begin{equation}
  \label{eq:Leps2}
  \mathcal L^\eps(t,x)= \(V(x)-T_2\(x,x(t)\)\) \varphi^\eps(t,x)
\end{equation}
is the nonlinear analogue of $L^\eps$ given by \eqref{def:Leps}. 
Duhamel's formula for $w^\eps$ reads
\begin{equation*}
  \begin{aligned}
     w^\eps(t+\tau) &= U^\eps(\tau)w^\eps(t)+ i\eps^{-1}\int_t^{t+\tau}
     U^\eps(t+\tau-s)\mathcal L^\eps(s)ds\\ 
&  \quad -i\l \eps^{\alpha_c-1}\int_t^{t+\tau}
U^\eps(t+\tau-s)\(|\psi^\eps|^{2\si}\psi^\eps 
  -|\varphi^\eps|^{2\si}\varphi^\eps\)(s)ds. 
  \end{aligned}
\end{equation*}
Introduce the following Lebesgue exponents:
\begin{equation*}
 \theta=\frac{2\si(2\si+2)}{2-(d-2)\si}\quad ;\quad
 q=\frac{4\si+4}{d\si}\quad ;\quad r=2\si+2. 
\end{equation*}
Then $(q,r)$ is admissible, and
\begin{equation*}
  \frac{1}{q'}=\frac{2\si}{\theta}+\frac{1}{q}\quad ;\quad
  \frac{1}{r'}=\frac{2\si}{r}+\frac{1}{r}. 
\end{equation*}
Let $t\ge 0$, $\tau>0$ and
$I=[t,t+\tau]$. Lemma~\ref{lem:scaledStrichartz} yields 
\begin{align*}
  \|w^\eps\|_{L^q(I;L^r)}&\lesssim \eps^{-1/q}\|w^\eps(t)\|_{L^2}+
  \eps^{-1-1/q}\|\mathcal L^\eps\|_{L^1(I;L^2)}\\
&\quad  +
  \eps^{\alpha_c-1-2/q}\left\lVert |\psi^\eps|^{2\si}\psi^\eps
  -|\varphi^\eps|^{2\si}\varphi^\eps\right\rVert_{L^{q'}(I;L^{r'})}.
\end{align*}
In view of the pointwise estimate
\begin{equation*}
  \left\lvert |\psi^\eps|^{2\si}\psi^\eps
  -|\varphi^\eps|^{2\si}\varphi^\eps\right\rvert \lesssim
  \(|w^\eps|^{2\si}+|\varphi^\eps|^{2\si}\) |w^\eps|,
\end{equation*}
we infer
\begin{equation}\label{eq:1410}
  \begin{aligned}
     \|w^\eps\|_{L^q(I;L^r)}&\lesssim \eps^{-1/q}\|w^\eps(t)\|_{L^2}+
  \eps^{-1-1/q}\|\mathcal L^\eps\|_{L^1(I;L^2)} \\
&\quad +
  \eps^{\alpha_c-1-2/q}\(\left\lVert
    w^\eps\right\rVert_{L^{\theta}(I;L^{r})}^{2\si}+\left\lVert
    \varphi^\eps\right\rVert_{L^{\theta}(I;L^{r})}^{2\si}\)\left\lVert
    w^\eps\right\rVert_{L^{q}(I;L^{r})}. 
  \end{aligned}
\end{equation}
Thanks to \cite{Ca-p}, we know that the rescaled functions for
$\psi^\eps$ and $\varphi^\eps$, are such that $u^\eps,u\in
C(\R;\Sigma)$, with estimates which are uniform in $\eps\in
]0,1]$. Typically,  for all $T>0$, there exists $C(T)$ independent of
$\eps$ such that
\begin{equation*}
  \|P u^\eps\|_{L^\infty([0,T];L^2)}
  +\|P u\|_{L^\infty([0,T];L^2)}  \le C(T),\quad P \in
  \{{\rm Id},\nabla,x\}. 
\end{equation*}
In terms of $\psi^\eps$ and $\varphi^\eps$, this yields
\begin{equation}\label{eq:1454}
  \|\mathcal P^\eps \psi^\eps\|_{L^\infty([0,T];L^2)}
  +\|\mathcal P^\eps  \varphi^\eps\|_{L^\infty([0,T];L^2)}\le
  C(T),\quad \mathcal P^\eps \in
  \{{\rm Id},A^\eps,B^\eps\}. 
\end{equation}
The formula
\eqref{eq:A} and Gagliardo--Nirenberg inequality yield, if $0\le \delta(p)<1$,
\begin{equation}\label{eq:GN}
  \|f\|_{L^p(\R^d)}\le
  \frac{C(p)}{\eps^{\delta(p)/2}}\left\lVert f\right\rVert_{L^2(\R^d)}^{1-\delta(p)}
  \left\lVert A^\eps(t)f\right\rVert_{L^2(\R^d)}^{\delta(p)} ,\quad \forall f\in
  H^1(\R^d),\quad \forall t\in \R.
\end{equation}
We infer that there exists $C(T)$ independent of $\eps$ such that 
\begin{equation}\label{est:psiphiLr}
  \lVert \psi^\eps(t)\rVert_{L^r(\R^d)} + \lVert
  \varphi^\eps(t)\rVert_{L^r(\R^d)} \le C(T) \eps^{-\delta(r)/2}
  ,\quad \forall t\in [0,T].
\end{equation}
Recalling that $I=[t,t+\tau]$, we deduce from \eqref{eq:1410}: 
\begin{align*}
  \|w^\eps\|_{L^q(I;L^r)}&\lesssim \eps^{-1/q}\|w^\eps(t)\|_{L^2}+
  \eps^{-1-1/q}\|\mathcal L^\eps\|_{L^1(I;L^2)} \\
&\quad +
  \eps^{\alpha_c-1-2/q}\tau^{2\si/\theta}\eps^{-\si\delta(r)}\left\lVert
    w^\eps\right\rVert_{L^{q}(I;L^{r})}.
\end{align*}
Since $(q,r)$ is admissible, we compute
\begin{equation}\label{eq:1419}
  \alpha_c-1-\frac{2}{q}-\si\delta(r)= \frac{d\si}{2}
  -\frac{2\si+2}{q}= 0,
\end{equation}
hence
\begin{equation}\label{eq:2208}
  \begin{aligned}
    \|w^\eps\|_{L^q(I;L^r)}&\lesssim \eps^{-1/q}\|w^\eps(t)\|_{L^2}+
  \eps^{-1-1/q}\|\mathcal L^\eps\varphi^\eps\|_{L^1(I;L^2)}\\
&\quad  
+
  \tau^{2\si/\theta}\left\lVert
    w^\eps\right\rVert_{L^{q}(I;L^{r})}.
 \end{aligned}
\end{equation}
Choosing $\tau$ sufficiently small, and repeating this manipulation a
finite number of times to cover the time interval $[0,T]$,
we obtain
\begin{equation}\label{eq:2131}
  \|w^\eps\|_{L^q([0,T];L^r)}\lesssim \eps^{-1/q}\|w^\eps\|_{L^1([0,T];L^2)}+
  \eps^{-1-1/q}\|\mathcal L^\eps\|_{L^1([0,T];L^2)} .
\end{equation}
Using Strichartz estimates again and \eqref{est:psiphiLr}, we have,
with $J=[0,t]$ and $0\le t\le T$, 
\begin{align*}
  \|w^\eps\|_{L^\infty(J;L^2)}&\lesssim
  \eps^{-1}\|\mathcal
  L^\eps\|_{L^1(J;L^2)}  +
  \eps^{\alpha_c-1-1/q}\left\lVert |\psi^\eps|^{2\si}\psi^\eps
  -|\varphi^\eps|^{2\si}\varphi^\eps\right\rVert_{L^{q'}(J;L^{r'})}\\
& \lesssim \|w^\eps\|_{L^1(J;L^2)}+
  \eps^{-1}\|\mathcal L^\eps\|_{L^1(J;L^2)} \\
&\quad +
  \eps^{\alpha_c-1-1/q}\eps^{-1-1/q-\si\delta(r)}\|\mathcal
  L^\eps\|_{L^1(J;L^2)}\\ 
&\lesssim \|w^\eps\|_{L^1(J;L^2)}+ \eps^{-1}\|\mathcal
L^\eps\|_{L^1(J;L^2)}, 
\end{align*}
where the last estimate stems from \eqref{eq:1419}. We have the
pointwise control
\begin{equation*}
  \lvert \mathcal L^\eps\rvert \lesssim  \lvert
  x-x(t)\rvert^3 \lvert \varphi^\eps(t,x)\rvert = \eps^{3/2} \eps^{-d/4}\(\lvert
  y\rvert^3 \lvert u(t,y)\rvert\)\Big|_{y=\frac{x-x(t)}{\sqrt\eps}}.
\end{equation*}
We infer
\begin{equation*}
 \eps^{-1}\|\mathcal L^\eps\|_{L^1([0,T];L^2(\R^d))}\lesssim \sqrt\eps
 \lVert \lvert y\rvert^3 u(t,y)\rVert_{L^1([0,T];L^2(\R^d))}, 
\end{equation*}
and the first part of Proposition~\ref{prop:bounded} follows from
Gronwall lemma. 

To establish a control of the $\mathcal H$-norm, we notice that in
view of \eqref{eq:A}, we have, for $\mathcal P^\eps\in \{A^\eps,B^\eps\}$,
\begin{equation*}
  \mathcal P^\eps \(|\phi^\eps|^{2\si}\phi^\eps\)\approx
  |\phi^\eps|^{2\si} \mathcal P^\eps \phi^\eps,
\end{equation*}
where the symbol ``$\approx$'' is here to recall the abuse of notation
when $\mathcal P^\eps =A^\eps$ (there should be two terms on the right
hand side, with coefficients). Lemma~\ref{lem:AB} shows that we have 
\begin{align*}
  \(i\eps\d_t +\frac{\eps^2}{2}\Delta -V\)A^\eps w^\eps &=
  \sqrt\eps\(\nabla V(x)-\nabla V\(x(t)\)\)w^\eps - A^\eps\mathcal
  L^\eps \\
&\quad +\l\eps^{\alpha_c} A^\eps \(|\psi^\eps|^{2\si}\psi^\eps
  -|\varphi^\eps|^{2\si}\varphi^\eps\). 
\end{align*}
The first term of the right hand side is controlled pointwise by
$C \eps |B^\eps w^\eps|$. The $L^2$-norm of the second term is
estimated by
\begin{equation*}
  \|A^\eps(t)\mathcal L^\eps (t)\|_{L^2(\R^d)}\lesssim
  \eps^{3/2}\(\lVert\lvert y\rvert^2 v(t,y)\rVert_{L^2(\R^d)}
  +\lVert\lvert y\rvert^3\nabla v(t,y)\rVert_{L^2(\R^d)}   \). 
\end{equation*}
Finally, we have
\begin{align}
 A^\eps \(|\psi^\eps|^{2\si}\psi^\eps
  -|\varphi^\eps|^{2\si}\varphi^\eps\)& \approx  |\psi^\eps|^{2\si}
  A^\eps\psi^\eps - |\varphi^\eps|^{2\si}A^\eps \varphi^\eps \notag \\
&\approx |w^\eps+\varphi^\eps|^{2\si}(A^\eps
w^\eps+A^\eps\varphi^\eps) - |\varphi^\eps|^{2\si}A^\eps
\varphi^\eps\notag \\
&\approx |w^\eps+\varphi^\eps|^{2\si}A^\eps w^\eps +
\(|w^\eps+\varphi^\eps|^{2\si}-|\varphi^\eps|^{2\si}\)A^\eps
\varphi^\eps.\label{eq:2123}
\end{align}
The first term of \eqref{eq:2123} is handled like in the first
step. For the second term, we have, since $\si> 1/2$, 
\begin{equation*}
 \left\lvert
   |w^\eps+\varphi^\eps|^{2\si}-|\varphi^\eps|^{2\si}\right\rvert
 \lesssim \( |w^\eps|^{2\si-1} + |\varphi^\eps|^{2\si-1}\)|w^\eps|. 
\end{equation*}
Following the same lines as for the $L^2$ estimate, we find
\begin{align*}
  \|A^\eps w^\eps\|_{L^q(I;L^r)}&\lesssim \eps^{-1/q}\|A^\eps(t)
  w^\eps(t)\|_{L^2}+
\eps^{-1/q}\|B^\eps
  w^\eps\|_{L^1(I;L^2)} \\
+
  \eps^{-1-1/q}&\|A^\eps \mathcal L^\eps\|_{L^1(I;L^2)}+
  \tau^{2\si/\theta}\left\lVert
    A^\eps w^\eps\right\rVert_{L^{q}(I;L^{r})} + \tau^{2\si/\theta} \left\lVert
     w^\eps\right\rVert_{L^{q}(I;L^{r})},
\end{align*}
by using the estimate
\begin{equation*}
  \|A^\eps (t)\varphi^\eps(t)\|_{L^r(\R^d)} \lesssim
  \eps^{-\delta(r)/2}\|A^\eps
  (t)\varphi^\eps(t)\|_{L^2(\R^d)}^{1-\delta(r)} \|A^\eps
  (t)^2\varphi^\eps(t)\|_{L^2(\R^d)}^{\delta(r)},
\end{equation*}
and the remark
\begin{equation*}
  \|A^\eps
  (t)^2\varphi^\eps(t)\|_{L^2(\R^d)}= \|\nabla^2 u(t)\|_{L^2(\R^d)}. 
\end{equation*}
Since $\si>1/2$, the nonlinearity $z\mapsto |z|^{2\si}z$ is twice
differentiable, and one can prove $u\in C(\R;H^2(\R^d))$ (\cite{Ca-p}). 
Using \eqref{eq:2131} and the same argument as in
the first step, we infer
\begin{align*}
  \|A^\eps w^\eps\|_{L^q(I;L^r)}&\lesssim \eps^{-1/q}\|A^\eps(t)
  w^\eps(t)\|_{L^2}+ \eps^{-1/q}\|B^\eps
  w^\eps\|_{L^1(I;L^2)} \\
&\quad +
  \eps^{-1-1/q}\|A^\eps \mathcal L^\eps\|_{L^1(I;L^2)}+
  \eps^{-1-1/q}  \left\lVert\mathcal
    L^\eps\right\rVert_{L^1(I;L^2)},
\end{align*}
hence, using Strichartz estimates again, 
\begin{align*}
  \|A^\eps w^\eps\|_{L^\infty(I;L^2)}&\lesssim \|A^\eps(t)
  w^\eps(t)\|_{L^2}+\|B^\eps
  w^\eps\|_{L^1(I;L^2)} +
  \eps^{-1}\|A^\eps \mathcal L^\eps\|_{L^1(I;L^2)}\\
&\quad +
  \eps^{1/q}\left\lVert
    A^\eps w^\eps\right\rVert_{L^{q}(I;L^{r})} +
  \eps^{1/q} \left\lVert 
     w^\eps\right\rVert_{L^{q}(I;L^{r})}\\
&\lesssim \|A^\eps(t)
  w^\eps(t)\|_{L^2}+\|B^\eps
  w^\eps\|_{L^1(I;L^2)} +
  \eps^{-1}\|A^\eps \mathcal L^\eps\|_{L^1(I;L^2)}
\\
&\quad +\eps^{-1}\| \mathcal L^\eps\|_{L^1(I;L^2)}+
  \eps^{1/q} \left\lVert 
     w^\eps\right\rVert_{L^{q}(I;L^{r})}.
\end{align*}
Since we have similar estimates for $B^\eps w^\eps$, we end up with
\begin{align*}
 \|A^\eps w^\eps\|_{L^\infty(J;L^2)} +  \|B^\eps
 w^\eps\|_{L^\infty(J;L^2)}&\lesssim
\|A^\eps w^\eps\|_{L^1(J;L^2)} +  \|B^\eps
 w^\eps\|_{L^1(J;L^2)}\\
&\quad  + \eps^{-1}\sum_{\mathcal P^\eps \in \{ {\rm
     Id},A^\eps,B^\eps\}} \|\mathcal P^\eps \mathcal
 L^\eps\|_{L^1(J;L^2)}\\
&\lesssim
\|A^\eps w^\eps\|_{L^1(J;L^2)} +  \|B^\eps
 w^\eps\|_{L^1(J;L^2)}+\sqrt\eps. 
\end{align*}
Proposition~\ref{prop:bounded} then follows from Gronwall lemma.

\section{Fully nonlinear case: proof of
  Theorem~\ref{theo:consistency}}
\label{sec:consistency} 

To prove Theorem~\ref{theo:consistency}, the strategy consists in
examining more carefully the dependence of the $L^\theta L^r$-norms
with respect to time in the previous proof. Also, since $(Exp)_4$
concerns only $u$, not $u^\eps$, we need a bootstrap argument in order
to use the same control for the error term $w^\eps$ as for the
approximate solution $\varphi^\eps$. This control carries on the
$L^r(\R^d)$-norms, for fixed $t$. By $(Exp)_1$, the relation
\begin{equation*}
  \|A^\eps
  (t)\varphi^\eps(t)\|_{L^2(\R^d)}= \|\nabla u(t)\|_{L^2(\R^d)},
\end{equation*}
and the modified
Gagliardo--Nirenberg inequality \eqref{eq:GN}, we have the following
estimate, for all time:
\begin{equation}\label{eq:phiexp}
  \|\varphi^\eps(t)\|_{L^r(\R^d)}\lesssim \eps^{-\delta(r)/2}
  e^{\kappa t}.
\end{equation}
We will use the following bootstrap argument:
\begin{equation}
  \label{eq:solong2}
  \|w^\eps(t)\|_{L^r(\R^d)}\le\eps^{-\delta(r)/2} e^{\kappa t},\quad t\in
  [0,T], 
\end{equation}
with the same constant $\kappa$ as in \eqref{eq:phiexp} to fix the ideas.
By Proposition~\ref{prop:bounded}, for any $T>0$ independent of
$\eps$, \eqref{eq:solong2} 
is satisfied provided $0<\eps\le \eps(T)$. By this argument only, it
may very well happen that $\eps(T)\to 0$ as $T\to +\infty$. The goal
of the bootstrap argument is to show that we can 
take $T^\eps = C\log\log\frac{1}{\eps}$ for some $C>0$ independent
of $\eps$, provided that $\eps$ is sufficiently small. 
\smallbreak 

The key step to analyze is the absorption argument, which made it
possible to infer \eqref{eq:2131} from \eqref{eq:2208}. We resume the
computations of \S\ref{sec:proofbounded} from the estimate
\eqref{eq:1410}. Rewrite this estimate with  $I=[t,t+\tau]$, $t,\tau\ge
0$:
\begin{equation*}
  \begin{aligned}
     \|w^\eps\|_{L^q(I;L^r)}&\lesssim \eps^{-1/q}\|w^\eps(t)\|_{L^2}+
  \eps^{-1-1/q}\|\mathcal L_\eps\|_{L^1(I;L^2)} \\
&\quad +
  \eps^{\alpha_c-1-2/q}\(\left\lVert
    w^\eps\right\rVert_{L^{\theta}(I;L^{r})}^{2\si}+\left\lVert
    \varphi^\eps\right\rVert_{L^{\theta}(I;L^{r})}^{2\si}\)\left\lVert
    w^\eps\right\rVert_{L^{q}(I;L^{r})}. 
  \end{aligned}
\end{equation*}
For simplicity, assume $\tau\le 1$: \eqref{eq:phiexp} and
\eqref{eq:solong2} yield, in view of \eqref{eq:1419},
\begin{equation*}
  \begin{aligned}
     \|w^\eps\|_{L^q(I;L^r)}& \le M\Big(\eps^{-1/q}\|w^\eps(t)\|_{L^2}+
  \eps^{-1-1/q}\|\mathcal L_\eps\|_{L^1(I;L^2)}\\
&\quad +
  \tau^{1/\theta} e^{2\si \kappa t}\left\lVert
    w^\eps\right\rVert_{L^{q}(I;L^{r})}\Big), 
  \end{aligned}
\end{equation*}
for some constant $M$ independent of $\eps$, $t\ge 0$ and $0\le
\tau\le 1$.
In order for the last term to be absorbed by the left hand side, we
have to assume
\begin{equation*}
  M \tau^{1/\theta} e^{2\si \kappa t}\le \frac{1}{2},\quad \text{that is},
  \quad \tau \le C e^{-C t}
\end{equation*}
for some $C$ independent of $\eps$, $t\ge 0$ and $0\le
\tau\le 1$. Proceeding with the same argument as in
\S\ref{sec:proofbounded}, we come up with:
\begin{align*}
  \|w^\eps\|_{L^\infty([0,t];L^2)}&\lesssim \int_0^t
  e^{Cs}\|w^\eps\|_{L^\infty([0,s];L^2)} ds + \eps^{-1}\int_0^t
  e^{Cs}\|\mathcal L_\eps(s)\|_{L^2}  ds\\
&\lesssim \int_0^t
  e^{Cs}\|w^\eps\|_{L^\infty([0,s];L^2)} ds + \sqrt\eps  e^{Ct},
\end{align*}
where we have used $(Exp)_3$. Gronwall lemma yields:
\begin{equation*}
  \|w^\eps\|_{L^\infty([0,t];L^2)}\lesssim \sqrt \eps\exp\(C\exp
  (Ct)\)\lesssim \sqrt \eps\exp\(\exp
  (2Ct)\).
\end{equation*}
Mimicking the computations of \S\ref{sec:proofbounded}, we have,
thanks to $(Exp)_4$ and so long as \eqref{eq:solong2} holds,
\begin{equation*}
  \|A^\eps w^\eps\|_{L^\infty([0,t];L^2)}+\|B^\eps
  w^\eps\|_{L^\infty([0,t];L^2)} \lesssim \sqrt \eps\exp\(\exp (Ct)\).
\end{equation*}
To conclude, we check that \eqref{eq:solong2} holds for $t\le c
\log\log\frac{1}{\eps}$, provided $c$ is sufficiently
small. Gagliardo--Nirenberg inequality \eqref{eq:GN} yields
\begin{align*}
  \|w^\eps(t)\|_{L^r(\R^d)}&\lesssim \eps^{-\delta(r)/2}
  \|w^\eps\|_{L^\infty([0,t];L^2)}^{1-\delta(r)}\|A^\eps
  w^\eps\|_{L^\infty([0,t];L^2)}^{\delta(r)}\\
&  \le \mathcal M
  \eps^{-\delta(r)/2} \sqrt \eps\exp\(\exp (Ct)\).
\end{align*}
Therefore, taking $\eps$ sufficiently small, \eqref{eq:solong2} holds
as long as
\begin{equation*}
  \mathcal M \sqrt \eps\exp\(\exp (Ct)\)\le e^{\kappa t}.
\end{equation*}
We check that for large $t$ and sufficiently small $\eps$, this
remains true for $t\le c \log \log \frac{1}{\eps}$, with $c$ possibly
small, but independent of $\eps\in ]0,\eps_0]$. This completes the
proof of Theorem~\ref{theo:consistency}.

\section{Ehrenfest time in the one-dimensional cubic case}
\label{sec:cubic}

As pointed out in the introduction, since we consider nonlinearities
of the form $z\mapsto |z|^{2\si}z$ with $\si\in \N$, the
one-dimensional cubic case is special. Not because it is integrable
(see Remark~\ref{rem:int}: \eqref{eq:u} is not completely integrable, 
unless no approximation is needed to describe the wave packets,
$\psi^\eps\equiv \varphi^\eps$),  but
because it is the only case where the nonlinearity is
$L^2$-subcritical, $\si<2/d$.  This case is in contrast with the
general case of energy-subcritical nonlinearities: without any other
assumption on $Q(t)$ than $Q\in C^\infty(\R;\R)\cap L^\infty(\R)$, it
seems that the only \emph{a priori} control that we have for $u$,
solution to \eqref{eq:u}, is 
\begin{equation}\label{eq:L2}
  \|u(t)\|_{L^2(\R^d)}=\|a\|_{L^2(\R^d)},\quad \forall t\in \R.
\end{equation}
A remarkable case where other \emph{a priori} estimates are available
is when $Q$ is constant, but in this case, $\psi^\eps\equiv
\varphi^\eps$. Otherwise, the most general reasonable
assumption seems to be $(Exp)_k$, which has been considered in the
previous section. Note also that if $d=1$,  the notations of
\S\ref{sec:proofbounded} become: 
\begin{equation*}
  \theta= \frac{8}{3}\quad ;\quad q=8\quad ;\quad r=4. 
\end{equation*}
So to improve the result of
Theorem~\ref{theo:consistency}, we assume $\si=d=1$ and start with the
crucial remark:
\begin{lemma}\label{lem:unifcubic}
  Suppose $\si=d=1$, and for $a\in L^2(\R)$, consider $u\in
  C(\R;L^2(\R))$ the solution to \eqref{eq:u}. Then 
  there exists $C$ such that
  \begin{equation*}
    \|u\|_{L^8([t,t+1];L^4(\R))}\le C \lVert a\rVert_{L^2(\R)},\quad \forall
    t\in \R.
  \end{equation*}
\end{lemma}
\begin{proof}
First, recall that since $\si=d=1$ and $a\in L^2(\R)$,
\eqref{eq:u} has a unique solution
\begin{equation*}
  u \in  C(\R;L^2(\R)) \cap L^8_{\rm loc}\(\R;L^4(\R)\).
\end{equation*}
In addition, \eqref{eq:L2} holds. Denoting 
\begin{equation*}
  W(t,x) = \frac{1}{2}V''\(x(t)\)x^2,
\end{equation*}
it has been established in \cite{Ca-p} that since $V''\in
L^\infty(\R;\R)$, uniform local Strichartz estimates are available
for the linear propagator. Following \cite{Fujiwara79,Fujiwara}, let
$U(t,s)$ be such that as $u(t,x) = U(t,s)u_0(x)$ is the solution to 
\begin{equation*}
  i \d_t u+\frac{1}{2}\Delta u  = W(t,x) u \quad ;\quad u(s,x)=u_0(x).
\end{equation*}
Then Lemma~\ref{lem:scaledStrichartz} remains true (with $\eps=1$)
when $U^\eps(t-s)$ is replaced with $U(t,s)$, $t,s\in \R$. 
\smallbreak

Let $t,\tau\ge 0$, with $\tau\le 1$, and denote
  $I=[t,t+\tau]$. Strichartz inequalities  yield:
  \begin{equation*}
    \|u\|_{L^8(I;L^4)}\le C(\tau)\|u(t)\|_{L^2}+C(\tau)\left\lVert
    |u|^2u\right\rVert_{L^{8/7}(I;L^{4/3})}. 
  \end{equation*}
In view of \eqref{eq:L2}, and
using H\"older inequality after the decomposition
\begin{equation*}
  \frac{3}{4} = \frac{3}{4}+\frac{1}{\infty} 
\quad ;\quad
  \frac{7}{8} = \frac{3}{8}+\frac{1}{2},
\end{equation*}
we infer
\begin{equation*}
  \|u\|_{L^8(I;L^4)}\le C(\tau) \|a\|_{L^2}+C(\tau)\sqrt\tau
  \|u\|_{L^8(I;L^4)}^3 .
\end{equation*}
Since $\tau\le 1$, we may assume that $C(\tau)$ does not depend on
$\tau$: 
\begin{equation*}
  \|u\|_{L^8(I;L^4)}\le C \|a\|_{L^2}+C\sqrt\tau
  \|u\|_{L^8(I;L^4)}^3 .
\end{equation*}
We use the following standard bootstrap argument, borrowed from \cite{BG99}:
\begin{lemma}[Bootstrap argument]\label{lem:boot}
Let $f=f(t)$ be a nonnegative continuous function on $[0,T]$ such
that, for every $t\in [0,T]$, 
\begin{equation*}
  f(t)\le M  + \delta f(t)^\theta,
\end{equation*}
where $M,\delta>0$ and $\theta >1$ are constants such that
\begin{equation*}
  M <\left(1-\frac{1}{\theta} \right)\frac{1}{(\theta \delta)^{1/(\theta
-1)}}\quad ;\quad f(0)\le  \frac{1}{(\theta \delta)^{1/(\theta
-1)}}.
\end{equation*}
Then, for every $t\in [0,T]$, we have
\begin{equation*}
  f(t)\le \frac{\theta}{\theta -1}M.
\end{equation*}
\end{lemma}
Lemma~\ref{lem:unifcubic} follows with $[t,t+1]$ replaced with
$[t,t+\tau]$ for $0<\tau\le \tau_0\ll 1$. We then
cover any  
interval of the form $[t,t+1]$ by a finite
number of intervals of length at most $\tau_0$, and
Lemma~\ref{lem:unifcubic} is proved.
\end{proof}

\begin{proof}[Proof of Theorem~\ref{theo:cubic}] Like in the previous
  section, we resume the proof of Proposition~\ref{prop:bounded}, and
  pay a more precise attention to the dependence of various constants
  upon time. We modify the bootstrap argument of
  \S\ref{sec:consistency}: in view of Lemma~\ref{lem:unifcubic},
  \eqref{eq:solong2} is replaced by 
  \begin{equation}
    \label{eq:solong3}
    \|w^\eps\|_{L^8([t,t+1];L^4(\R))}\le  \eps^{-1/8} \lVert
    a\rVert_{L^2(\R)},\quad \forall 
    t\in [0,T].
  \end{equation}
By Proposition~\ref{prop:bounded}, for any $T>0$ independent of
$\eps$, \eqref{eq:solong3} remains true provided $0<\eps\le
\eps(T)$. Keeping the notations of \S\ref{sec:proofbounded}, we have:
\begin{equation*}
  \theta= \frac{8}{3}\quad ;\quad q=8\quad ;\quad r=4. 
\end{equation*}
With $t\ge 0$, $\tau \in ]0,1]$ and $I=[t,t+\tau]$, \eqref{eq:1410} becomes
  \begin{align}
     \|w^\eps\|_{L^8(I;L^4)}&\lesssim \eps^{-1/8}\|w^\eps(t)\|_{L^2}+
  \eps^{-1-1/8}\|\mathcal L_\eps\|_{L^1(I;L^2)} \notag \\
&\quad +
  \eps^{1/4}\(\left\lVert
    w^\eps\right\rVert_{L^{8/3}(I;L^{4})}^{2}+\left\lVert
    \varphi^\eps\right\rVert_{L^{8/3}(I;L^{4})}^{2}\)\left\lVert
    w^\eps\right\rVert_{L^{8}(I;L^{4})}\notag \\
& \lesssim \eps^{-1/8}\|w^\eps(t)\|_{L^2}+
  \eps^{-1-1/8}\|\mathcal L_\eps\|_{L^1(I;L^2)} \notag \\
&\quad +
  \eps^{1/4}\tau^{1/4}\(\left\lVert
    w^\eps\right\rVert_{L^{8}(I;L^{4})}^{2}+\left\lVert
    \varphi^\eps\right\rVert_{L^{8}(I;L^{4})}^{2}\)\left\lVert
    w^\eps\right\rVert_{L^{8}(I;L^{4})}\notag \\
& \lesssim \eps^{-1/8}\|w^\eps(t)\|_{L^2}+
  \eps^{-1-1/8}\|\mathcal L_\eps\|_{L^1(I;L^2)}  +\tau^{1/4}\left\lVert
    w^\eps\right\rVert_{L^{8}(I;L^{4})},\label{eq:1626}
  \end{align}
where we have used Lemma~\ref{lem:unifcubic} and
\eqref{eq:solong3}. Choosing $\tau$ sufficiently small and independent
of $t$, we come up with
\begin{align*}
  \|w^\eps\|_{L^\infty([0,t];L^2)}&\lesssim \|w^\eps\|_{L^1([0,t];L^2)}
  + \eps^{-1} \|\mathcal L_\eps\|_{L^1([0,t];L^2)}\\
&\lesssim
  \|w^\eps\|_{L^1([0,t];L^2)} 
  + \sqrt\eps \int_0^te^{Cs}ds,
\end{align*}
by $(Exp)_3$. Gronwall lemma yields
\begin{equation*}
  \|w^\eps\|_{L^\infty([0,t];L^2)}\lesssim \sqrt\eps e^{Ct}.
\end{equation*}
Back to \eqref{eq:1626}, we infer, with $\tau\ll 1$,
\begin{equation*}
  \|w^\eps\|_{L^8(I;L^4)}\lesssim \eps^{1/4}e^{Ct}.
\end{equation*}
Therefore, there exists $c>0$ such that \eqref{eq:solong3} holds for
$T=c\log\frac{1}{\eps}$ provided $\eps$ is sufficiently small, hence
the first part of Theorem~\ref{theo:cubic}.
\smallbreak

It is then quite straightforward to infer the estimates in $\mathcal
H$, by rewriting the end of the proof of 
Proposition~\ref{prop:bounded}, with \eqref{eq:solong3} in mind. 
\end{proof}

\section{Nonlinear superposition}
\label{sec:superposition}

\subsection{General considerations}

The proof of Proposition~\ref{prop:superposition}  and
Theorem~\ref{theo:superposition}  follows the same lines as the
proof of Proposition~\ref{prop:bounded} and
Theorem~\ref{theo:cubic}. The main difference comes 
from the way one deals with the nonlinearity, since new terms
appear. These terms come from the nonlinear interaction between the two
profiles $\varphi^\eps_1$ and $\varphi_2^\eps$. Denote $w^\eps =
\psi^\eps -\varphi_1^\eps-\varphi^\eps_2$. It solves
\begin{equation*}
  i\eps\d_t w^\eps +\frac{\eps^2}{2}\Delta w^\eps = Vw^\eps -\mathcal
  L^\eps+\l\mathcal N^\eps\quad ;\quad w^\eps_{\mid t=0}=0,
\end{equation*}
where we have now
\begin{equation*}
  \mathcal L^\eps(t,x)= \(V(x)-T_2\(x,x(t)\)\)\(
  \varphi_1^\eps(t,x)+\varphi_2^\eps(t,x)\),
\end{equation*}
and
\begin{equation*}
\mathcal N^\eps=\eps^{\alpha_c} \(\lvert w^\eps +
  \varphi_1^\eps+\varphi_2^\eps\rvert^{2\si}(w^\eps +
  \varphi_1^\eps+\varphi_2^\eps)- \lvert 
  \varphi_1^\eps\rvert^{2\si}\varphi_1^\eps-\lvert
  \varphi_2^\eps\rvert^{2\si}\varphi_2^\eps\). 
\end{equation*}
Decompose $\mathcal N^\eps$ as the sum of a semilinear term and an
interaction source term: $\mathcal N^\eps=\mathcal N^\eps_S+\mathcal
N^\eps_I$, where
\begin{align*}
  \mathcal N^\eps_S &= \eps^{\alpha_c}\(\lvert w^\eps +
  \varphi_1^\eps+\varphi_2^\eps\rvert^{2\si}(w^\eps +
  \varphi_1^\eps+\varphi_2^\eps)-\lvert 
  \varphi_1^\eps+\varphi_2^\eps\rvert^{2\si}(
  \varphi_1^\eps+\varphi_2^\eps)\),\\
\mathcal N^\eps_I&= \eps^{\alpha_c}\(\lvert 
  \varphi_1^\eps+\varphi_2^\eps\rvert^{2\si}(
  \varphi_1^\eps+\varphi_2^\eps)-\lvert 
  \varphi_1^\eps\rvert^{2\si}\varphi_1^\eps-\lvert
  \varphi_2^\eps\rvert^{2\si}\varphi_2^\eps\). 
\end{align*}
We see that the term $\mathcal N^\eps_S$ is the exact analogue of the
nonlinear term in \eqref{eq:restecrit}, where we have simply replaced
$\varphi^\eps$ with $\varphi_1^\eps+\varphi_2^\eps$. We can thus
repeat the proofs of Proposition~\ref{prop:bounded} and
Theorem~\ref{theo:cubic}, respectively, up to the control of the
new source term $\mathcal N^\eps_I$ (the linear source term $\mathcal
L^\eps$ is treated as before). More precisely, we have to estimate
\begin{equation*}
 \frac{1}{\eps} \lVert \mathcal N^\eps_I\rVert_{L^1([0,t];L^2(\R^d))}. 
\end{equation*}
The first remark consists in noticing that if $\si$ is an integer, 
$\mathcal N^\eps_I$ can be estimated (pointwise) by a sum of terms of
the form 
\begin{equation*}
\eps^{\alpha_c}\lvert \varphi_1^\eps\rvert^{\ell_1}\times \lvert
\varphi_2^\eps\rvert^{\ell_2},\quad \ell_1,\ell_2\ge 1, \
\ell_1+\ell_2=2\si+1. 
\end{equation*}
To be more precise, we have the control, for fixed time,
\begin{equation*}
\frac{1}{\eps}  \lVert \mathcal N^\eps_I(t)\rVert_{L^2(\R^d)}\lesssim
  \eps^{d\si/2} \sum_{\ell_1,\ell_2\ge 1, \
\ell_1+\ell_2=2\si+1 } \left\lVert
\lvert\varphi_1^\eps\rvert^{\ell_1}\times \lvert 
\varphi_2^\eps\rvert^{\ell_2}\right\rVert_{L^2(\R^d)}.
\end{equation*}
We will see below why the right hand side must be expected to be
small, when integrated with respect to time. 
We need to estimate
\begin{equation*}
 \eps^{d\si/2} \left\lVert \(\varphi_1^\eps\)^{\ell_1}
\(\varphi_2^\eps\)^{\ell_2}\right\rVert_{L^2(\R^d)} = \left\lVert
 u_1^{\ell_1}\(t,x -\frac{x_1(t)-x_2(t)}{\sqrt\eps}\)
 u_2^{\ell_2}(t,x)\right\rVert_{L^2(\R^d_x)},
\end{equation*}
with $\ell_1,\ell_2\ge 1$, $\ell_1+\ell_2=2\si+1$. 
 We have the following lemma:
 \begin{lemma}\label{lem:1655} 
Suppose $d\le 3$, and $\si$ is an integer. Let $T\in\R$, $0<\gamma<1/2$, and 
\begin{equation}
  \label{eq:ieps}
  I^\eps(T)=\{t\in[0,T],\;\;|x_1(t)-x_2(t)|\le \eps^\gamma\}.
\end{equation}
  Then, for all $k>d/2$,
  \begin{equation*}
    \frac{1}{\eps}\int_0^T\|\mathcal N_I^\eps(t)\|_{\Sigma_\eps}dt\lesssim
 \(M_{k+2}(T)\)^{2\si+1} \left(T 
   \eps^{k(1/2-\gamma)} +  |I^\eps(T)|\right)e^{CT},
  \end{equation*}
where 
$\displaystyle  M_k(T)= \sup \left\{\|\<x\>^\alpha \partial_x^\beta
  u_j\|_{L^\infty([0,T];L^2(\R^d)}; \quad j\in \{1,2\},\quad
  |\alpha|+|\beta|\le k
\right\}.$
\end{lemma}
 \begin{proof} 
We observe that for $\eta\in \R^d$,
\begin{equation*}
  \sup_{x\in\R^d} \left(\<x\>^{-1}\<x-\eta\>^{-1}\right)\lesssim
  \frac{1}{|\eta|}. 
\end{equation*}
With $\eta^\eps(t) = \frac{x_1(t)-x_2(t)}{\sqrt\eps}$, we infer
(forgetting the sum over $\ell_1,\ell_2$), 
\begin{align*}
  \frac{1}{\eps}&\int_{[0,T]\setminus I^\eps(T)} \|\mathcal
  N_I^\eps(t)\|_{L^2(\R)}dt \lesssim \\
&\lesssim
  \int_{[0,T]\setminus I^\eps(T)}\left\lVert \<x-\eta^\eps(t)\>^{-k}\<x\>^{-k}
    \<x-\eta^\eps(t)\>^{-k}
    u_1^{\ell_1}\(t,x-\eta^\eps(t)\)u_2^{\ell_2}(t,x)\right\rVert_{L^2} \\
&\lesssim \left\lVert \<x\>^k
  u_1^{\ell_1}\right\rVert_{L^\infty([0,T];L^4)}\left\lVert 
\<x\>^{k} u_2^{\ell_2}\right\rVert_{L^\infty([0,T];L^4)}
\int_{[0,T]\setminus 
  I^\eps(T)} \frac{dt}{\lvert \eta^\eps(t)\rvert^k}.
\end{align*}
We have, for $j\in \{1,2\}$,
\begin{align*}
 \left\lVert \<x\>^k
  u_j^{\ell_j}\right\rVert_{L^\infty([0,T];L^4)}&\le \left\lVert \<x\>^k
  u_j\right\rVert_{L^\infty([0,T];L^4)} \left\lVert 
  u_j\right\rVert_{L^\infty([0,T]\times \R^d)}^{\ell_j-1}\\
&\lesssim \left\lVert \<x\>^k
  u_j\right\rVert_{L^\infty([0,T];H^1)} \left\lVert 
  u_j\right\rVert_{L^\infty([0,T];H^k)}^{\ell_1-1}
\lesssim M_{k+1}(T)^{\ell_j},
\end{align*}
where we have used $H^1(\R^d)\subset L^4(\R^d)$ since $d\le 3$. On the
other hand,
\begin{equation*}
\int_{[0,T]\setminus 
  I^\eps(T)} \frac{dt}{\lvert \eta^\eps(t)\rvert^k} \lesssim
\int_{[0,T]\setminus 
  I^\eps(T)} \frac{\eps^{k/2}}{\lvert x_1(t)-x_2(t)\rvert^k}dt\lesssim
\eps^{k(1/2-\gamma)}T. 
\end{equation*}
On $I^\eps(T)$, we simply estimate
\begin{align*}
 \frac{1}{\eps}\int_{I^\eps(T)} \|\mathcal
  N_I^\eps(t)\|_{L^2(\R)}dt &  \lesssim
  \|u_1\|_{L^\infty([0,T]\times\R)}^{\ell_1}
  \|u_2\|_{L^\infty([0,T]\times\R)}^{\ell_2-1}
  \|u_2\|_{L^1(I^\eps(T);L^2(\R))}\\
&\lesssim M_k(T)^{2\si}
\lvert I^\eps(T)\rvert \lVert u_2\|_{L^\infty([0,T];L^2(\R))}\\
&\lesssim M_k(T)^{2\si+1}
\lvert I^\eps(T)\rvert  .
\end{align*}
The $L^2$ estimate follows, without exponentially growing factor. This
factor appears when dealing with the $\Sigma_\eps$-norm. Typically,
\begin{align*}
  \|\eps \nabla \varphi_j^\eps(t)\|_{L^2(\R^d)}&\lesssim \sqrt\eps \|\nabla
  u_j(t)\|_{L^2(\R^d)} + \lvert \xi_j(t)\rvert \lVert
  u_j(t)\rVert_{L^2(\R^d)},\\
\|x\varphi_j^\eps(t)\|_{L^2(\R^d)}&\lesssim \sqrt\eps \|x
  u_j(t)\|_{L^2(\R^d)} + \lvert x_j(t)\rvert \lVert
  u_j(t)\rVert_{L^2(\R^d)}.
\end{align*}
The result then follows from the above computations, and
Lemma~\ref{lem:traj}.  
\end{proof}
 At this stage, the main difficulty is to estimate the length of
 $I^\eps(t)$. We do this in two cases: bounded $t$, and large time
 when $d=1$. 
 \subsection{Nonlinear superposition in finite time}
In the proof of Proposition~\ref{prop:bounded}, we have only used the
fact that $u^\eps\in C(\R;\Sigma)$, with estimates which are
independent of $\eps$. Recall that in the case of a single wave
packet, $\psi^\eps$ and $u^\eps$ are related through
\eqref{eq:chginc}: in the case of two wave packets, there is no such
natural rescaling. So in the case of two initial wave packets, we are
not able to prove uniform estimates for $\psi^\eps$, like in
\eqref{eq:1454}. Even to prove Proposition~\ref{prop:superposition},
which is the analogue of Proposition~\ref{prop:bounded}, we need to
use a bootstrap argument. We know that for $j\in \{1,2\}$,
\begin{equation*}
  \lVert
  \varphi_j^\eps(t)\rVert_{L^r(\R^d)} \le C(T) \eps^{-\delta(r)/2}
  ,\quad \forall t\in [0,T].
\end{equation*}
The bootstrap argument is of the form:
\begin{equation*}
  \lVert w^\eps(t)\rVert_{L^r(\R^d)} \le C(T) \eps^{-\delta(r)/2}
  ,\quad \forall t\in [0,T],
\end{equation*} 
with the same constant $C(T)$ if we wish. Repeating the computations
of \S\ref{sec:proofbounded}, we first have, for $t\in [0,T]$ and so
long as the above condition holds,
\begin{equation*}
  \|w^\eps\|_{L^\infty([0,t];L^2)}\lesssim \frac{1}{\eps}\|\mathcal
  L^\eps\|_{L^1([0,T];L^2)} + \frac{1}{\eps}\|\mathcal
  N_I^\eps\|_{L^1([0,T];L^2)} .
\end{equation*}
As we have seen in \S\ref{sec:linear}, $(\eps\nabla w^\eps,x w^\eps)$
solves a system which is formally analogous to the system satisfied by
$(A^\eps w^\eps,B^\eps w^\eps)$. Therefore, under the bootstrap
condition, we come up with
 \begin{equation*}
  \|w^\eps\|_{L^\infty([0,t];\Sigma_\eps)}\lesssim \frac{1}{\eps}\|\mathcal
  L^\eps\|_{L^1([0,T];\Sigma_\eps)} + \frac{1}{\eps}\|\mathcal
  N_I^\eps\|_{L^1([0,T];\Sigma_\eps)} .
\end{equation*}
We easily estimate
\begin{equation*}
  \frac{1}{\eps}\|\mathcal
  L^\eps\|_{L^1([0,T];\Sigma_\eps)} \lesssim \sqrt \eps,
\end{equation*}
so in view of Lemma~\ref{lem:1655}, the point is to estimate the
length of $I^\eps(T)$.  
 \begin{lemma}\label{lem:1652}
   For $T>0$ (independent of $\eps$), we have
   \begin{equation*}
     \lvert I^\eps(T)\rvert= \O\(\eps^{\gamma}\),
   \end{equation*}
where $I^\eps(T)$ is defined by \eqref{eq:ieps}.
 \end{lemma}
 \begin{proof}
 The key remark is that 
 since $(x_1,\xi_1)\not=(x_2,\xi_2)$, the trajectories
 $x_1(t)$ and $x_2(t)$ may cross only in isolated points: by
 uniqueness, if $x_1(t)=x_2(t)$, then $\dot x_1(t)\not = \dot
 x_2(t)$. Therefore, 
 there is only a finite numbers of such points in the interval
 $[0,T]$:
 \begin{equation*}
\(x_1(\cdot)-x_2(\cdot)\)^{-1}(0)\cap [0,T] = \{t_j\}_{1\le j\le
     J},\quad \text{where }J=J(T). 
 \end{equation*}
If we had $J=\infty$, then by compactness of $[0,T]$, a
subsequence of $(t_j)_j$ would converge to some $\tau\in [0,T]$, with
 $x_1(\tau)=x_2(\tau)$. By uniqueness for the Hamiltonian flow,
$\dot x_1(\tau)\not = \dot x_2(\tau)$: $\tau$ cannot be the limit of
times where $x_1(t_j)=x_2(t_j)$. 
\smallbreak

By uniqueness for the Hamiltonian flow, continuity and compactness,
there exists $\delta>0$ such that
\begin{equation*}
  \inf\{ \lvert \dot x_1(t)-\dot x_2(t)\rvert\ ;\  t\in \mathcal
  I(\delta,T)\}=m >0,\quad \text{where }
\mathcal  I(\delta,T)  =\bigcup_{j=1}^J [t_j-\delta,t_j+\delta],
\end{equation*}
and there exists $\eps(\delta,T)>0$ such that for $\eps\in
]0,\eps(\delta,T)]$, $I^\eps(T)\subset \mathcal I(\delta,T)$.
\smallbreak

Let $t\in I^\eps(T)\cap [t_j-\delta,t_j+\delta]$. Taylor's formula
yields
\begin{equation*}
  x_1(t)-x_2(t) = x_1(t_j)-x_2(t_j) +(t-t_j)\(\dot x_1(\tau)-\dot
  x_2(\tau)\),\quad \tau \in [t_j-\delta,t_j+\delta].
\end{equation*}
We infer
\begin{equation*}
  \eps^\gamma \ge \lvert x_1(t)-x_2(t)\rvert \ge \lvert t-t_j\rvert m,
\end{equation*}
and Lemma~\ref{lem:1652} follows. 
 \end{proof}
Back to the bootstrap argument, we infer
\begin{equation*}
 \|w^\eps\|_{L^\infty([0,t];\Sigma_\eps)}\lesssim \sqrt\eps +
 \eps^{k(1/2-\gamma)}+\eps^\gamma.  
\end{equation*}
Fix $\gamma\in ]0,1/2[$.  By taking $k$ sufficiently large in
Lemma~\ref{lem:1655}, this yields 
\begin{equation*}
 \|w^\eps\|_{L^\infty([0,t];\Sigma_\eps)}\lesssim \eps^\gamma.  
\end{equation*}
Gagliardo--Nirenberg inequality yields
\begin{equation*}
  \|w^\eps(t)\|_{L^r}\lesssim
  \eps^{-\delta(r)}\|w^\eps(t)\|_{L^2}^{1-\delta(r)} \|\eps \nabla
  w^\eps(t)\|_{L^2}^{\delta(r)} \lesssim \eps^{-\delta(r)}
  \|w^\eps\|_{L^\infty([0,t];\Sigma_\eps)} \lesssim
  \eps^{\gamma-\delta(r)}. 
\end{equation*}
To close the argument, we note
\begin{equation*}
  \eps^{\gamma-\delta(r)}\ll \eps^{-\delta(r)/2} \text{ provided }
  \eps\ll1\text{ and } \gamma>\frac{\delta(r)}{2}.
\end{equation*}
The last condition is equivalent to $\gamma>\frac{d\si}{4\si+4}$,
which is compatible with $\gamma<1/2$ since the nonlinearity is
energy-subcritical. 
 \subsection{Nonlinear superposition for large time}
 
 Things become more complicated when $T$ is large. We first
 need to control $M_k$: this is achieved  assuming 
 $(Exp)_k$, and we have 
 \begin{equation*}
   M_k(t)\lesssim e^{Ct}.
 \end{equation*}
The main point is to estimate  $| I_\eps|$. This is achieved
thanks to the following proposition, whose proof relies heavily on the
fact that the space variable is one-dimensional. 
\begin{proposition}\label{prop:2traj} 
Under the assumptions of Theorem~\ref{theo:superposition}, there exist
 $C,C_0>0$ independent of $\eps$ such that
\begin{equation*}
  |I_\eps(t)|\lesssim \eps^\gamma e^{C_0t} |E_1-E_2|^{-2},\quad
  0\le t\le C\log\frac{1}{\eps}.
\end{equation*}
\end{proposition}
\begin{proof}[Proof of Theorem~\ref{theo:superposition}]
  Before proving Proposition~\ref{prop:2traj}, we show why this is
  enough to infer 
  Theorem~\ref{theo:superposition}. By Lemma~\ref{lem:1655}, we have,
  if $(Exp)_k$ is satisfied,
  \begin{equation*}
    \frac{1}{\eps}\|\mathcal N_I^\eps\|_{L^1([0,t];\Sigma_\eps)}\lesssim
    e^{Ct}\( t \eps^{(k-2)(1/2-\gamma)} + \eps^\gamma e^{C_0t}\)\lesssim
    \(\eps^{(k-2)(1/2-\gamma)} + \eps^\gamma \)e^{Ct}. 
  \end{equation*}
Optimizing in $\gamma$, we require $(k-2)(1/2-\gamma)=\gamma$, that is 
\begin{equation*}
  \gamma = \frac{k-2}{2k-2}. 
\end{equation*}
We can thus resume the bootstrap arguments as in \S\ref{sec:consistency} and
\S\ref{sec:cubic}, respectively. The key is to notice that this works like
in the previous paragraph, since 
\begin{equation*}
  \gamma = \frac{k-2}{2k-2}> \frac{\si}{4\si+4}\quad (k\ge 4). 
\end{equation*}
This yields Theorem~\ref{theo:superposition}
\end{proof}

\begin{proof}[Proof of Proposition~\ref{prop:2traj}]
  We consider 
 $J^\eps(t)$  an interval of maximal length included in $I^\eps(t)$ and
 $N^\eps(t)$ the number of such intervals. The result comes from the
 estimate
 $$|I^\eps(t)|\le N^\eps(t) \times \max |J^\eps(t)|, $$
 with 
\begin{align}\label{J}
 |J^\eps(t)| & \lesssim  \eps^\gamma e^{Ct}\lvert E_1-E_2\rvert^{-1},\\
 \label{N}
 N^\eps(t) & \lesssim   t e^{2Ct} \lvert E_1-E_2\rvert^{-1}\lesssim
 e^{3Ct} \lvert E_1-E_2\rvert^{-1}  .
 \end{align}
 We first prove  prove \eqref{J}.  Let
$t_1,t_2\in J^\eps(t)$. There exists  $t^*\in [t_1,t_2]$ such that 
 $$\left|\left(x_1(t_1)-x_2(t_1)\right)
   -\left(x_1(t_2)-x_2(t_2)\right)\right|=
|t_2-t_1\rvert \left\lvert\xi_1(t^*)-\xi_2(t^*)\right|,$$ 
whence
 $$ |t_1-t_2|\le |\xi_1(t^*)-\xi_2(t^*)|^{-1} \times 2 \eps^\gamma.$$
 On the other hand,
 \begin{equation*}
  |\xi_1(t^*)-\xi_2(t^*)|\ge \left\lvert \lvert\xi_1(t^*)\rvert-\lvert
    \xi_2(t^*)\rvert\right\rvert
\ge \frac{\left| |\xi_1(t^*)|^2-
|\xi_2(t^*)|^2\right|}{|\xi_1(t^*)|+|\xi_2(t^*)|}.  
 \end{equation*}
Using 
\begin{align*} 
|\xi_1(t^*)|+|\xi_2(t^*)| & \lesssim  e^{Ct},\\
|\xi_1(t^*)|^2-|\xi_2(t^*)|^2 & = 
2\left(E_1-E_2-V(x_1(t^*))+V(x_2(t^*))\right),\\ 
\left|V(x_1(t^*))-V(x_2(t^*))\right| & \le   \eps^\gamma {\rm e}^{Ct},
\end{align*}
 we get
 $$||\xi_1(t^*)|^2-|\xi_2(t^*)|^2|\gtrsim |E_1-E_2 |-\eps^\gamma e^{Ct},$$
 whence
 $$ |t_1-t_2|  \lesssim \eps^\gamma{\rm e}^{Ct}  |E_1-E_2|^{-1},$$
 provided $\eps^\gamma e^{Ct}\ll 1$.
 \smallbreak

Let us now prove \eqref{N}. We use that as $t$ is large, $N^\eps(t)$ is
comparable to the number of distinct intervals of maximal size where
$|x_1(t)-x_2(t)|\ge \eps^\gamma$. 
 We consider $J'_\eps=[t'_1,t'_2]$ such an interval . We have 
 $$|x_1(t'_1)-x_2(t'_1)|=|x_1(t'_2)-x_2(t'_2)|=\eps^\gamma, \text{ and
   }\forall t\in[t'_1,t'_2],\quad|x_1(t)-x_2(t)|\ge\eps^\gamma.$$
Therefore, for $t\in[t'_1,t'_2]$, the quantity $x_1(t)-x_2(t)$ has a
constant sign: we suppose that $x_1(t)-x_2(t)$ is positive. We then have  
$$\xi_1(t'_1)-\xi_2(t'_1)>0\;\;{\rm and}\;\;\xi_1(t'_2)-\xi_2(t'_2)<0.$$
Using the exponential control of $V'(x_j(t))$ for $j\in\{1,2\}$, we obtain
$$\left(\xi_1(t'_1)-\xi_2(t'_1)\right)-
\left(\xi_1(t'_2)-\xi_2(t'_2)\right)\lesssim{\rm
  e}^{Ct}|t'_1-t'_2|.$$ 
We write
\begin{align*}
  \xi_1(t'_1)-\xi_2(t'_1)=|\xi_1(t'_1)-\xi_2(t'_1)|&\ge
  \frac{\left||\xi_1(t'_1)|^2- |\xi_2(t'_1)|^2\right|}{|\xi_1(t'_1)|+
    |\xi_2(t'_1)|}\\
&\gtrsim
  e^{-Ct}\left||\xi_1(t'_1)|^2-|\xi_2(t'_1)|^2\right|,\\
-\xi_1(t'_2)+\xi_2(t'_2)=|\xi'_1(t_2)-\xi_2(t'_2)|& \ge
\frac{\left||\xi_1(t'_2)|^2- |\xi_2(t'_2)|^2\right|}{|\xi_1(t'_2)|+
|\xi_2(t'_2)|}\\
&\gtrsim
e^{-Ct}\left||\xi_1(t'_2)|^2-|\xi_2(t'_2)|^2\right|. 
\end{align*}

Besides, in view of 
\begin{align*}
\frac{1}{2} \left( |\xi_1(t'_1)|^2-|\xi_2(t'_2)|^2\right) & =  E_1-E_2
-V(x_1(t'_1))+V(x_2(t'_1))\\ 
 & =  E_1-E_2-V'(x^*)\left[x_1(t'_1)-x_2(t'_1)\right]
 \end{align*}
with $x^*\in[x_2(t'_1),x_1(t'_1)]$, we have
$$\left|V'(x^*)\left[x_1(t'_1)-x_2(t'_1)\right]\right|\lesssim e^{Ct}
\left[x_1(t'_1)-x_2(t'_1)\right]\lesssim \eps^\gamma e^{Ct}.$$
Therefore, if $\eps^\gamma e^{Ct}\ll 1$, we have  $E_1-E_2>0$ and
$$\frac{1}{2}
\left| |\xi_1(t'_1)|^2-|\xi_2(t'_)|^2\right|\ge \frac{1}{2}( E_1-E_2).$$
The same  holds for  $t'_2$, which yields
$$\left(\xi_1(t'_1)-\xi_2(t'_1)\right)-\left(\xi_1(t'_2)-
\xi_2(t'_2)\right)\gtrsim e^{-Ct}(E_1-E_2),$$ 
whence the existence of a constant $c>0$ such that 
$$|t'_1-t'_2|\ge c e^{-2CT} (E_1-E_2)\;\;{\rm and } \;\;|J'_\eps|\ge
c\, {\rm e}^{-2CT} (E_1-E_2).$$ 
The number  $\tilde N^\eps(t)$ of intervals of the type $J'_\eps$ satisfies
$$\tilde N^\eps(t)\times c e^{-2Ct} (E_1-E_2)\le t$$
whence the second point of the claim.
\end{proof}

\bibliographystyle{amsplain}
\bibliography{biblio}

\providecommand{\bysame}{\leavevmode\hbox to3em{\hrulefill}\thinspace}
\providecommand{\MR}{\relax\ifhmode\unskip\space\fi MR }
\providecommand{\MRhref}[2]{%
  \href{http://www.ams.org/mathscinet-getitem?mr=#1}{#2}
}
\providecommand{\href}[2]{#2}
\begin{thebibliography}{10}

\bibitem{AblowitzClarkson}
M.~J. Ablowitz and P.~A. Clarkson, \emph{Solitons, nonlinear evolution
  equations and inverse scattering}, London Mathematical Society Lecture Note
  Series, vol. 149, Cambridge University Press, Cambridge, 1991.

\bibitem{BG99}
H.~Bahouri and P.~G{\'e}rard, \emph{High frequency approximation of solutions
  to critical nonlinear wave equations}, Amer. J. Math. \textbf{121} (1999),
  no.~1, 131--175.

\bibitem{BGP99}
D.~Bambusi, S.~Graffi, and T.~Paul, \emph{Long time semiclassical approximation
  of quantum flows: a proof of the {E}hrenfest time}, Asymptot. Anal.
  \textbf{21} (1999), no.~2, 149--160.

\bibitem{BR01}
J.~M. Bily and D.~Robert, \emph{The semi-classical {V}an {V}leck formula.
  {A}pplication to the {A}haronov-{B}ohm effect}, Long time behaviour of
  classical and quantum systems ({B}ologna, 1999), Ser. Concr. Appl. Math.,
  vol.~1, World Sci. Publ., River Edge, NJ, 2001, pp.~89--106.

\bibitem{BJ00}
J.~C. Bronski and R.~L. Jerrard, \emph{Soliton dynamics in a potential}, Math.
  Res. Lett. \textbf{7} (2000), no.~2-3, 329--342.

\bibitem{CaBook}
R.~Carles, \emph{Semi-classical analysis for nonlinear {S}chr\"odinger
  equations}, World Scientific Publishing Co. Pte. Ltd., Hackensack, NJ, 2008.

\bibitem{Ca-p}
\bysame, \emph{Nonlinear {S}chr{\"o}dinger equation with time dependent
  potential}, preprint. Archived as {\tt ar{X}iv:0910.4893}, 2009.

\bibitem{CFG}
R.~Carles, C.~Fermanian, and I.~Gallagher, \emph{On the role of quadratic
  oscillations in nonlinear {S}chr{\"o}dinger equations}, J. Funct. Anal.
  \textbf{203} (2003), no.~2, 453--493.

\bibitem{CM04}
R.~Carles and L.~Miller, \emph{Semiclassical nonlinear {S}chr\"odinger
  equations with potential and focusing initial data}, Osaka J. Math.
  \textbf{41} (2004), no.~3, 693--725.

\bibitem{CR97}
M.~Combescure and D.~Robert, \emph{Semiclassical spreading of quantum wave
  packets and applications near unstable fixed points of the classical flow},
  Asymptot. Anal. \textbf{14} (1997), no.~4, 377--404.

\bibitem{CR06}
\bysame, \emph{Quadratic quantum {H}amiltonians revisited}, Cubo \textbf{8}
  (2006), no.~1, 61--86.

\bibitem{CR07}
\bysame, \emph{A phase-space study of the quantum {L}oschmidt echo in the
  semiclassical limit}, Ann. Henri Poincar\'e \textbf{8} (2007), no.~1,
  91--108.

\bibitem{DGPS}
F.~Dalfovo, S.~Giorgini, L.~P. Pitaevskii, and S.~Stringari, \emph{Theory of
  {B}ose-{E}instein condensation in trapped gases}, Rev. Mod. Phys. \textbf{71}
  (1999), no.~3, 463--512.

\bibitem{FGJS04}
J.~Fr{\"o}hlich, S.~Gustafson, B.~L.~G. Jonsson, and I.~M. Sigal,
  \emph{Solitary wave dynamics in an external potential}, Comm. Math. Phys.
  \textbf{250} (2004), no.~3, 613--642.

\bibitem{Fujiwara79}
D.~Fujiwara, \emph{A construction of the fundamental solution for the
  {S}chr\"odinger equation}, J. Analyse Math. \textbf{35} (1979), 41--96.

\bibitem{Fujiwara}
\bysame, \emph{Remarks on the convergence of the {F}eynman path integrals},
  Duke Math. J. \textbf{47} (1980), no.~3, 559--600.

\bibitem{GS07}
Zhou Gang and I.~M. Sigal, \emph{Relaxation of solitons in nonlinear
  {S}chr\"odinger equations with potential}, Adv. Math. \textbf{216} (2007),
  no.~2, 443--490.

\bibitem{PG96}
P.~G{\'e}rard, \emph{Oscillations and concentration effects in semilinear
  dispersive wave equations}, J. Funct. Anal. \textbf{141} (1996), no.~1,
  60--98.

\bibitem{GMMP}
P.~G{\'e}rard, P.~A. Markowich, N.~J. Mauser, and F.~Poupaud,
  \emph{Homogenization limits and {W}igner transforms}, Comm. Pure Appl. Math.
  \textbf{50} (1997), no.~4, 323--379.

\bibitem{GV85}
J.~Ginibre and G.~Velo, \emph{Scattering theory in the energy space for a class
  of nonlinear {S}chr\"odinger equations}, J. Math. Pures Appl. (9) \textbf{64}
  (1985), no.~4, 363--401.

\bibitem{H80}
G.~A. Hagedorn, \emph{Semiclassical quantum mechanics. {I}. {T}he {$\hbar
  \rightarrow 0$} limit for coherent states}, Comm. Math. Phys. \textbf{71}
  (1980), no.~1, 77--93.

\bibitem{HJ00}
G.~A. Hagedorn and A.~Joye, \emph{Exponentially accurate semiclassical
  dynamics: propagation, localization, {E}hrenfest times, scattering, and more
  general states}, Ann. Henri Poincar\'e \textbf{1} (2000), no.~5, 837--883.

\bibitem{HJ01}
\bysame, \emph{A time-dependent {B}orn-{O}ppenheimer approximation with
  exponentially small error estimates}, Comm. Math. Phys. \textbf{223} (2001),
  no.~3, 583--626.

\bibitem{HZ07}
J.~Holmer and M.~Zworski, \emph{Slow soliton interaction with delta
  impurities}, J. Mod. Dyn. \textbf{1} (2007), no.~4, 689--718.

\bibitem{JFGS06}
B.~L.~G. Jonsson, J.~Fr{\"o}hlich, S.~Gustafson, and I.~M. Sigal, \emph{Long
  time motion of {NLS} solitary waves in a confining potential}, Ann. Henri
  Poincar\'e \textbf{7} (2006), no.~4, 621--660.

\bibitem{JP01}
C.~Josserand and Y.~Pomeau, \emph{Nonlinear aspects of the theory of
  {B}ose-{E}instein condensates}, Nonlinearity \textbf{14} (2001), no.~5,
  R25--R62.

\bibitem{KT}
M.~Keel and T.~Tao, \emph{Endpoint {S}trichartz estimates}, Amer. J. Math.
  \textbf{120} (1998), no.~5, 955--980.

\bibitem{Keraani02}
S.~Keraani, \emph{Semiclassical limit for a class of nonlinear {S}chr\"odinger
  equations with potential}, Comm. Part. Diff. Eq. \textbf{27} (2002), no.~3-4,
  693--704.

\bibitem{KeraaniAA}
\bysame, \emph{Semiclassical limit for nonlinear {S}chr\"odinger equation with
  potential. {II}}, Asymptot. Anal. \textbf{47} (2006), no.~3-4, 171--186.

\bibitem{kha09}
U~Al Khawaja, \emph{Soliton localization in {B}ose--{E}instein condensates with
  time-dependent harmonic potential and scattering length}, J. Phys. A: Math.
  Theor. \textbf{42} (2009), 265206.

\bibitem{KMR09}
J.~Krieger, Y.~Martel, and P.~Rapha{\"e}l, \emph{Two-soliton solutions to the
  three-dimensional gravitational {H}artree equation}, Comm. Pure Appl. Math.
  \textbf{62} (2009), no.~11, 1501--1550.

\bibitem{LionsPaul}
P.-L. Lions and T.~Paul, \emph{Sur les mesures de {W}igner}, Rev. Mat.
  Iberoamericana \textbf{9} (1993), no.~3, 553--618.

\bibitem{L86}
R.~G. Littlejohn, \emph{The semiclassical evolution of wave packets}, Phys.
  Rep. \textbf{138} (1986), no.~4-5, 193--291.

\bibitem{Ozawa91}
T.~Ozawa, \emph{Long range scattering for nonlinear {S}chr\"odinger equations
  in one space dimension}, Comm. Math. Phys. \textbf{139} (1991), 479--493.

\bibitem{P97}
T.~Paul, \emph{Semi-classical methods with emphasis on coherent states},
  Quasiclassical methods ({M}inneapolis, {MN}, 1995), IMA Vol. Math. Appl.,
  vol.~95, Springer, New York, 1997, pp.~51--88.

\bibitem{Rob-p}
D.~Robert, \emph{On the {H}erman-{K}luk semiclassical approximation}, preprint.
  Archived as {\tt ar{X}iv:0908.0847}, 2009.

\bibitem{Rou-p}
V.~Rousse, \emph{Semiclassical simple initial value representations}, preprint.
  Archived as {\tt ar{X}iv:0904.0387}, 2009.

\bibitem{SHB07}
V.~N. Serkin, A.~Hasegawa, and T.~L. Belyaeva, \emph{Nonautonomous solitons in
  external potentials}, Phys. Rev. Lett. \textbf{98} (2007), no.~7, 074102.

\bibitem{SR09}
T.~Swart and V.~Rousse, \emph{A mathematical justification for the
  {H}erman-{K}luk propagator}, Comm. Math. Phys. \textbf{286} (2009), no.~2,
  725--750.

\bibitem{Yajima87}
K.~Yajima, \emph{Existence of solutions for {S}chr\"odinger evolution
  equations}, Comm. Math. Phys. \textbf{110} (1987), 415--426.

\end{thebibliography}

\end{document}